\documentclass{mcom-l}

\usepackage{graphicx}
\usepackage{amsmath}
\usepackage{mathtools}
\usepackage{cancel}
\usepackage{xcolor}
\usepackage{soul}
\usepackage[ruled,vlined]{algorithm2e}
\usepackage{ulem}
\usepackage{caption}
\newtheorem{theorem}{Theorem}[]
\newtheorem{assump}[]{Assumption}
\newtheorem{lemma}[subsection]{Lemma}

\theoremstyle{definition}

\newtheorem{example}[]{Example}

\theoremstyle{remark}

\newcommand{\mup}{\boldsymbol{\mu}}
\newcommand{\ub}{\textbf{u}}
\newcommand{\vb}{\textbf{v}}
\newcommand{\xib}{\boldsymbol{\xi}}

\newcommand{\x}{\mathbf{x}}
\newcommand{\y}{\mathbf{y}}
\newcommand{\z}{\mathbf{z}}
\DeclareMathOperator*{\esssup}{ess\,sup}

\newcommand{\cond}{\boldsymbol{\sigma}}
\newcommand{\adve}{\boldsymbol{b}}
\newcommand{\spacedim}{d}

\newcommand{\advep}{\adve_{\mup}}
\newcommand{\forcep}{f_{\mup}}
\newcommand{\condp}{\cond_{\mup}}
\newcommand{\gp}{g_{\mup}}
\newcommand{\up}{u_{\mup}}

\newcommand{\solmanifold}{\mathcal{S}}
\newcommand{\pdim}{p}

\newcommand{\solmanifoldh}{\solmanifold^{h}}
\newcommand{\fomdim}{N_{h}}
\newcommand{\ufom}{\ub^{h}}
\newcommand{\ufomp}{\ufom_{\mup}}
\newcommand{\ufompi}{\ufom_{\mup_{i}}}
\newcommand{\Nu}{\mathcal{V}}
\newcommand{\nub}{\boldsymbol{\nu}}
\newcommand{\diam}{\text{diam}}

\newcommand{\mindim}{n_{\min}}
\newcommand{\rdim}{n}
\newcommand{\urom}{\ub^{\rdim}}
\newcommand{\uromp}{\ub^{\rdim}_{\mup}}
\newcommand{\urompi}{\ub^{\rdim}_{\mup_{i}}}

\newcommand{\ntrain}{N_{\text{train}}}
\newcommand{\ntest}{N_{\text{test}}}
\newcommand{\weight}{\mathbf{W}}
\newcommand{\bias}{\mathbf{b}}

\newcommand{\dlrom}{\Phi}
\newcommand{\decoder}{\Psi}
\newcommand{\nndecoder}{\Psi}
\newcommand{\nnencoder}{\Psi'}

\newcommand{\nnphi}{\phi}
\def\tildedecoder{\tilde{\decoder}}
\def\tildeencoder{\tilde{\decoder}'}
\newcommand{\encoder}{\decoder'}
\newcommand{\transcoder}{\encoder_{\mu}}
\newcommand{\mui}{\mup_{i}}
\newcommand{\ui}{\ufom_{\mui}}
\newcommand{\uri}{\urom_{\mui}}
\newcommand{\lossf}{\text{Loss}}
\newcommand{\loss}{\mathcal{L}}

\newcommand{\review}[1]{#1}
\newcommand{\final}[1]{#1}

\begin{document}

\title[]{A Deep Learning approach to Reduced Order Modelling of Parameter dependent Partial Differential Equations}

\author{Nicola R. Franco}
\address{MOX, Department of Mathematics, Politecnico di Milano, Italy}
\email{nicolarares.franco@polimi.it}

\author{Andrea Manzoni}
\address{MOX, Department of Mathematics, Politecnico di Milano, Italy}
\email{andrea1.manzoni@polimi.it}

\author{Paolo Zunino}
\address{MOX, Department of Mathematics, Politecnico di Milano, Italy}
\email{paolo.zunino@polimi.it}


\date{}

\begin{abstract}
Within the framework of parameter dependent PDEs, we develop a constructive approach based on Deep Neural Networks for the efficient approximation of the parameter-to-solution map. The research is motivated by the limitations and drawbacks of state-of-the-art algorithms, such as the Reduced Basis method, when addressing problems that show a slow decay in the Kolmogorov $n$-width. Our work is based on the use of deep autoencoders, which we employ for encoding and decoding a high fidelity approximation of the solution manifold.
\final{To provide guidelines for the design of deep autoencoders}, we consider a nonlinear version of the Kolmogorov $n$-width over which we base the concept of a minimal latent dimension. We show that the latter is intimately related to the topological properties of the solution manifold, and we provide theoretical results with particular emphasis on second order elliptic PDEs, \final{characterizing the minimal dimension and the approximation errors of the proposed approach}.
\final{The theory presented is further supported by numerical experiments, where we compare the proposed approach} with classical POD-Galerkin reduced order models. In particular, we consider parametrized advection-diffusion PDEs, and we test the methodology in the presence of strong transport fields, singular terms and stochastic coefficients.
\end{abstract}
\maketitle

\section*{Introduction}
\label{sec:intro}
In many areas of science, such as physics, biology and engineering, phenomena are modeled in terms of
Partial Differential Equations (PDEs) that exhibit dependence on one or multiple parameters. As an example, consider the stationary advection-diffusion equation below,
\begin{equation*}
\label{advection:diffusion}
\begin{cases}
-\textnormal{div}(\condp \nabla u) + \advep\cdot\nabla u = \forcep & \textnormal{in }\Omega,\\
u=\gp & \textnormal{on }\partial\Omega,
\end{cases}
\end{equation*}
where $\Omega\subset\mathbb{R}^{\spacedim}$ is a bounded domain and $\mup$ a vector parameter taking values in a suitable parameter space $\Theta\subset\mathbb{R}^{\pdim}$. For each $\mup\in\Theta$, we assume the above to admit a unique solution $u_{\mup}$, to be sought within a given Hilbert space $(V, ||\cdot||)$.
\newline\indent
In some cases, one is not interested in computing the PDE solution for a single fixed $\mup\in\Theta$, but rather for an ensemble of parameter values. In general, this corresponds to exploring the so-called \textit{solution manifold} $\solmanifold:=\{\up\}_{\mup\in\Theta}$ \cite{fink,lassila}. The map $\mup\to\up$ is known under many equivalent names such as the \textit{parametric map} \cite{schwab}, the parameter-to-state map \cite{hoffmann} or the solution map \cite{ohlberger}.
Approximating the parametric map in a highly-efficient way is a challenging task that can be encountered in several contexts, from optimal control problems with parametric PDEs constraints \cite{binev} to multiscale fluid mechanics \cite{knezevic}, or Bayesian inversion and uncertainty quantification \cite{ghattas}. In all these cases, the main drawback is represented by the computational cost entailed by traditional PDE solvers. In fact, despite their accuracy, each query of a numerical scheme such as the Finite Element Method (FEM) implies a computational cost that easily becomes unsustainable in many query applications, where computations are supposed to be carried out within a very short amount of time.
\newline\indent
One possibility is then to replace Full Order Models (FOMs) with cheaper surrogate models, namely Reduced Order Models (ROMs). ROMs originate from the need of alleviating the computational burden of FOMs at the price of a negligible compromise in terms of accuracy. During the last decades, several successful model reduction techniques have been developed, such as the Reduced Basis method \cite{quarteroni} and others. However, the majority of these ROMs heavily relies on linear projection techniques, thus limiting significantly the spectrum of possible applications. Indeed, ROMs based on linear reduction methods encounter substantial difficulties whenever the solution manifold has a so-called Kolmogorov $\rdim$-width \cite{kolmogorov} that decays slowly with $\rdim$.
The Kolmogorov $\rdim$-width is a quantity that measures the degree of accuracy by which a set can be approximated using linear subspaces of dimension $\rdim$, namely
\begin{equation}
\label{eq:kolmogorov}
d_{\rdim}(\solmanifold) := \inf_{\substack{\\V_{\rdim}\subset V, \\\text{dim}(V_{\rdim})=\rdim}}\sup_{u\in\solmanifold}\;\inf_{\vb\in V_{\rdim}} ||u-v||.
\end{equation}
If $d_{\rdim}(\solmanifold)$ decays slowly with $\rdim$, then projection-based ROMs can reach meaningful accuracies only for large values of $n$, which in turn leads to expensive models. We point out that this phenomenon is far from being uncommon. As a matter of fact, the slow decay on $d_{\rdim}(\solmanifold)$ is typical of time-dependent transport-dominated problems, even under fairly simple circumstances \cite{greif,ohlberger}. The same is also true for stationary and purely diffusive problems, provided that the parameters enter in a highly nonlinear and possibly spatially localized way. \final{The interested reader can find a simple yet remarkable example of this fact in the Appendix, Example \ref{example:kolmogorov-decay}}.

In order to tackle these drawbacks, we propose a novel approach based on Deep Neural Networks (DNNs) \cite{schmidhuber} that naturally accounts for possible nonlinearities in the solution manifold. Our construction is mostly inspired by the recent advancements in nonlinear approximation theory, e.g. \cite{cohen2,daubechies,schwab}, and the increasing use of deep learning techniques for parametrized PDEs and \final{operator learning}, as in \cite{hesthaven3,fresca,kutyniok,karniadakis}.

\subsection*{Our contribution}
\label{sec:our}
The purpose of the present work is to provide alternative ROM techniques for parametrized PDEs that are able to overcome the drawbacks implied by the slow decay of the Kolmogorov $\rdim$-width. In particular, we focus on nonintrusive ROMs where the solution map is approximated by a deep neural network $\dlrom$. This idea has been recently investigated both theoretically, as in \cite{kutyniok, mishra3, schwab}, and practically, e.g. \cite{fresca,geist}. By now, the drawbacks posed by this approach are mainly practical: it is often unclear how the network architecture should be designed and which optimization strategies are better suited for the purpose. Also, we lack the understanding of the possible ways the nonlinearities in the DNN should be exploited in order to make the most out of it. Here, we wish to partially answer these questions and provide a constructive way of designing such $\Phi$. 

The key idea is to break the problem into two parts. First, we seek for a low-dimensional representation of the solution manifold, which we obtain by training a deep autoencoder \cite{hinton}, $\nndecoder\circ\nnencoder$.
The encoder, $\nnencoder$, is used to map the solution manifold into a reduced feature space $\mathbb{R}^{\rdim}$, while the decoder serves for the reconstruction task. Here we see a clear analogy with the Nonlinear Kolmogorov $\rdim$-width as defined in DeVore et al. \cite{devore}. There, the authors define
\begin{equation*}
\delta_{\rdim}(\solmanifold) 
:= 
\inf_{
    \substack{
    \encoder\in\mathcal{C}(\solmanifold,\;\mathbb{R}^{n})\\
    \decoder\in\mathcal{C}(\mathbb{R}^{n},\;V)}}
    \;\;\sup_{u\in\solmanifold}\;\;||u-\decoder(\encoder(u))||,
\end{equation*}
as a nonlinear counterpart of $d_{\rdim}(\solmanifold)$. In light of this, we introduce the concept of \textit{minimal latent dimension}, denoted as $\mindim(\solmanifold)$, which we define as the smallest $\rdim$ for which $\delta_{\rdim}(\solmanifold)=0$. By choosing this particular $\mindim(\solmanifold)$ as latent dimension for the autoencoder, we are then able to perform a significant model reduction.

Once the autoencoder has been trained, we exploit the encoder $\nnencoder$ in order to represent each solution $\up$ through a low-dimensional vector $\uromp\in\mathbb{R}^{n}$. We then train a third network $\nnphi:\Theta\to\mathbb{R}^{n}$ to learn the \textit{reduced map} $\mup\to\uromp$. In this way, by connecting the architectures of $\nnphi$ and $\nndecoder$ we obtain the complete model, $\dlrom:=\nndecoder\circ\nnphi$, which we later term as DL-ROM (Deep Learning based Reduced Order Model, in the same spirit of previous works \cite{fresca,fresca2}).
\newline
\newline
The novelty of our contribution is twofold. First, we develop a new constructive way of using neural networks to approximate the solution map and we test it on some numerical examples. Second, we prove theoretical results that motivate the choice of the ROM dimension. Indeed, despite the popularity of autoencoders, e.g. \cite{fresca,lee,mucke,zhu}, the choice of the latent dimension $n$ is often handled by trial and error. In contrast, we establish precise bounds on $n$ thanks to a rigorous theoretical analysis. 

More precisely, in Theorems 1 and 2, we investigate the link between the minimal latent dimension $\mindim(\solmanifold)$ and the topological properties of $\solmanifold$. In Theorem 3 we explicitly bound $\mindim(\solmanifold)$ in terms of the dimensionality of the parameter space. In particular, we show that $\mindim(\solmanifold)\le 2p+1$ as soon as the parametric map is Lipschitz continuous. \review{The theory is then applied to the case of second order elliptic PDEs, in Theorem 4, where we demonstrate how the parameters directly affect the value of the minimal latent dimension. \final{Finally, in Theorem 5, we bound the model complexity in terms of the ROM accuracy, deriving suitable error estimates that are later confirmed experimentally.}}
\newline
\newline
The paper is organized as follows. In Section \ref{sec:background} we introduce our general framework and briefly recall the driving ideas of linear reduction.
In Section \ref{sec:nonlinear} we move to the nonlinear case, where we establish a solid theoretical background for the construction of the DL-ROM, with particular emphasis on minimal representations and parametrized PDEs. In Section \ref{sec:learning} we dive into the details of our deep learning approach, thereby discussing the general construction and \review{its numerical properties}. In Section \ref{sec:experiments} we present some numerical results and assess the proposed methodology. Finally, to make the paper self-contained, auxiliary mathematical results are reported in the Appendix.
\newpage

\section{General background}
\label{sec:background}
Within the present Section we formally introduce the problem of reduced order modelling for parametrized PDEs. For later comparison, we also take the chance to recall the linear reduction technique known as Principal Orthogonal Decomposition \cite{luo,quarteroni}. In the remainder of the paper, we make use of elementary notions coming from the areas of Functional Analysis, Numerical Analysis and Topology. We respectively refer to \cite{adams,evans}, \cite{quarteroni-fem} and \cite{hocking}.

\subsection{Reduced Order Models for parametrized PDEs}
\label{sec:roms}
We are given a parameter space $\Theta\subset\mathbb{R}^\pdim$, a Hilbert state space $(V, ||\cdot||)$ and \final{parameter dependent operators $a_{\mup}: V\times V\to\mathbb{R}$ and $f_{\mup}: V\to\mathbb{R}$. For each $\mup\in\Theta$ we consider the variational problem}
\begin{equation}
\label{eq:ppde}
u\in V:\quad a_{\mup}(u, v)=f_{\mup}(v)\quad\forall v\in V.
\end{equation}
We assume the problem to be well-posed, so that for each $\mup\in\Theta$ there exists a unique solution $u=\up\in V$.
Our interest is to define a ROM that is able to approximate the parametric map $\mup\to\up$ efficiently.
In general, the workflow goes as follows. First, one chooses a FOM, which we here assume to be based on Galerkin projections. This corresponds to fixing a so-called \textit{high-fidelity discretization}, that is a finite dimensional subspace $V_{h}\subset V$, dim$(V_{h})=\fomdim$, used to replace the original trial space. Having chosen a basis for $V_{h}$, say $\{\varphi_{i}\}_{i=1}^{\fomdim}$, for each $\mup\in\Theta$ one turns equation \eqref{eq:ppde} into the (discrete) problem
\begin{equation}
    \label{eq:disc}
    \ufomp=[\textrm{u}_{\mup,1}^{h},\dots,\textrm{u}_{\mup,\fomdim}^{h}]\in\mathbb{R}^{\fomdim}:\quad a_{\mup}\left(\sum_{i=1}^{n}\textrm{u}_{\mup,i}^{h}\varphi_{i},\; v\right)=f_{\mup}(v)\quad\forall v\in V_{h}.
\end{equation}
The main purpose of the high-fidelity discretization is to reframe the original problem within a finite dimensional setting, without particular care on the computational cost (for now). Regarding the choice of $V_{h}$, we make the following assumption.

\begin{assump}
\label{assump:high-fidelity}
For any $\varepsilon>0$ there exists $V_{h}:=\textnormal{span}\{\varphi_{i}\}_{i}^{\fomdim}\subset V$ such that
$$\sup_{\mup\in\Theta}\left|\left|\up-\sum_{i=1}^{\fomdim}\textnormal{$\textrm{u}_{\mup,i}^{h}$}\varphi_{i}\right|\right|<\varepsilon$$
that is, the FOM accuracy can be bounded independently on the value of $\mup\in\Theta$.
\end{assump}
\noindent\newline
The above is a very common assumption in the literature, see e.g. \cite{kutyniok}, that allows us to formally replace $V$ with $V_{h}$. The objective now becomes that of learning the map $\mup\to\ufomp$ in a way that reduces the FOM cost. In particular, the construction of the ROM consists in finding a suitable map $\Phi:\mathbb{R}^{p}\to\mathbb{R}^{\fomdim}$ for which $\Phi(\mup)\approx\ufomp$. To do so, the common practice is to make extensive use of the FOM during a preliminary offline stage, which results in the collection of the so-called \textit{snapshots}, $\{\mup_{i}, \ufompi\}_{i}\subset\mathbb{R}^{p}\times\mathbb{R}^{\fomdim}$. These snapshots are then processed in order to build the ROM. In this sense, the identification of $\Phi$ can be seen as a problem of Statistical Learning, as argued in \cite{kutyniok}.
The way $\Phi$ is defined from the data is what characterizes each ROM, its efficiency and accuracy.

\subsection{Methods based on linear projections}
\label{sec:pod}
Many state-of-the-art ROMs are built upon the use of linear reduction techniques, which are known to work particularly well for second order elliptic PDEs with affine coefficients \cite{cohen}. The idea is the following. Having fixed a high-fidelity discretization, one considers the (discretized) solution manifold $\solmanifoldh:=\{\ufomp\}_{\mup\in\Theta}$ and tries to approximate it using linear subspaces. This translates into fixing a reduced dimension $\rdim\in\mathbb{N}$ and searching for the orthonormal matrix $\mathbf{V}\in\mathbb{R}^{\fomdim\times\rdim}$ that minimizes the errors $||\ufomp - \mathbf{V}\mathbf{V}^{T}\ufomp||$. In practice, the identification of such $\mathbf{V}$ is done empirically by exploiting the aforementioned snapshots $\{\mup_{i}, \ufompi\}_{i=1}^{N}$, which is often achieved through the so-called Principal Orthogonal Decomposition (POD). In short, this latter approach considers all the FOM snapshots as columns of a matrix, $\mathbf{U}:=[\ub^{h}_{\mup_{1}},,\ub^{h}_{\mup_{N}}]$ and computes its singular value decomposition
$$\mathbf{U} = \tilde{\mathbf{U}}\Sigma\mathbf{W}^{T}$$
where $\Sigma=$diag$(\sigma_{1},..,\sigma_{N})$ with $\sigma_{1}\ge...\ge\sigma_{N}\ge0$.
Then, in the POD approach, $\mathbf{V}$ is defined by extracting the first $n$ columns of $\tilde{\mathbf{U}}$. It is well-known that this choice of $\mathbf{V}$ is optimal -in some sense- over the training sample $\{\ufompi\}_{i=1}^{N}$. We also mention that, while this version of the POD considers $\mathbb{R}^{\fomdim}$ as a Euclidean space, slight modifications allow to account for different (e.g. energy) norms.

Once $\mathbf{V}$ has been built, the solution manifold is projected onto the reduced space $\mathbb{R}^{n}$, and each FOM solution is associated with the corresponding low-dimensional representation, $\uromp:=\mathbf{V}^{T}\ufomp$.
\final{To be operational, the ROM then needs to implement a suitable algorithm that approximates the correspondence $\mup\to\uromp$. If we represent the latter as a map $\phi:\Theta\to\mathbb{R}^{n}$, then the ROM approximation of high-fidelity solutions can be written as $\Phi(\mup):=\mathbf{V}\phi(\mup)\approx\ufomp$.
Within the literature, this step has been handled in multiple ways. In the Reduced Basis method \cite{lassila,quarteroni}, particularly in the so-called POD-Galerkin method, $\phi$ is defined intrusively by projecting and solving equation \eqref{eq:disc} onto span$(\mathbf{V})$.} Depending on the parametrization and on the type of PDE, this procedure may turn to be too expensive, which is why several alternatives have been proposed, see e.g. \cite{barrault,manzoni,shah}. \final{Nonintrusive approaches for defining $\phi$ include Gaussian process regression \cite{hesthaven1}, polynomial chaos expansions \cite{jiang}, neural networks \cite{hesthaven3, hesthaven2} and others}.

Nevertheless, because of the linear approximation, these ROMs encounter substantial difficulties as soon as $\solmanifold$ has a Kolmogorov $n$-width, see equation \eqref{eq:kolmogorov}, that decays slowly. In fact,
$$ \sup_{\mup\in\Theta}||\ufomp-\mathbf{V}\phi(\mup)||\;\ge\;\sup_{\mup\in\Theta}||\ufomp-\mathbf{V}\mathbf{V}^{T}\ufomp||
\;\ge\;d_{n}(\solmanifoldh)\;\ge\; d_{n}(\solmanifold)-\varepsilon,$$
where $\varepsilon>0$ is the accuracy of the high-fidelity discretization. Therefore, if $d_{n}(\solmanifold)$ decays slowly, one may be forced to consider large values of $\rdim$, which in turn makes $\phi$ more expensive and harder to identify.
As we argue in the next Section, one possible solution to this problem is given by nonlinear reduction techniques. However, despite this being a promising direction, only a few steps have been made so far, e.g. \cite{bhattacharjee,fresca,lee}.

\section{Nonlinear Dimensionality Reduction}
\label{sec:nonlinear}

In the present Section we formalize the idea of using nonlinear reduction techniques for the compression of the solution manifold. We start by introducing all concepts and results in an abstract fashion. Only at the end, Section \ref{sec:nonlinear-param-pdes}, we rephrase the content in terms of parametrized PDEs.

\subsection{Nonlinear Kolmogorov $\rdim$-width and minimal latent dimension}
\label{sec:nl-kw}

Within this Section, we consider an abstract setting where $(V, ||\cdot||)$ is a Hilbert space and $\solmanifold\subset V$ a generic subset. In particular, $\solmanifold$ needs not to be the solution manifold of a parametrized PDE and the theory is presented regardless of a possible discretization. We address the problem of finding a low-dimensional representation of $\solmanifold$ while minimizing the reconstruction error. 

When $V$ is finite-dimensional, the linear reduction described in Section \ref{sec:pod} performs an encoding of $\solmanifold$ via the map $\ub\to\mathbf{V}^{T}\ub=:\ub^{n}\in\mathbb{R}^{\rdim}$, where $\rdim$ is the reduced dimension; the set is then recovered through $\ub^{n}\to \mathbf{V}\ub^{n}\approx\ub$. Therefore, a possible generalization to the nonlinear case is to substitute $\mathbf{V}^{T}$ with some \textit{encoder} $\encoder:\solmanifold\to\mathbb{R}^{\rdim}$ and $\mathbf{V}$ with a \textit{decoder} $\decoder:\mathbb{R}^{\rdim}\to V$. Of note, this is also an approach that easily extends to infinite-dimensional settings. Depending on the restrictions that we impose on $\encoder$ and $\decoder$, different reconstruction accuracies can be obtained.

Here, we only require $\encoder$ and $\decoder$ to be continuous. This, naturally gives rise to the optimization problem below,
\begin{equation}
\label{eq:nonlinear-kolmogorov}
\delta_{\rdim}(\solmanifold) 
:= 
\inf_{
    \substack{
    \encoder\in\mathcal{C}(\solmanifold,\;\mathbb{R}^{n})\\
    \decoder\in\mathcal{C}(\mathbb{R}^{n},\;V)}}
    \;\;\sup_{u\in\solmanifold}\;\;||u-\decoder(\encoder(u))||,
\end{equation}
where $\mathcal{C}(X, Y)$ denotes the collection of all continuous maps from $X$ to $Y$. 
As we mentioned in the Introduction, the above corresponds to the (continuous) Nonlinear Kolmogorov $n$-width as defined in \cite{devore}. It is clear that $d_{\rdim}(\solmanifold)\ge\delta_{\rdim}(\solmanifold)$. Also, $\delta_{\rdim}(\solmanifold)$ is nonincreasing in $\rdim$, which reflects the fact that better approximations can be achieved in higher dimensional spaces. However, in the context of reduced order modelling, smaller values of $\rdim$ are often preferable, as they allow for less expensive models.
In this sense, whenever there exists a smallest dimension $\mindim$ that allows for an arbitrarily accurate reduction, i.e. $\delta_{\mindim}(\solmanifold)=0$, we may want to focus on that one. For this reason, we introduce the notation
$$\mindim(\solmanifold):=\min\{\rdim\in\mathbb{N}\;|\; \delta_{\rdim}(\solmanifold)=0\},\vspace{1em}$$ were we adopt the convention $\min\emptyset=+\infty$. We refer to $\mindim(\solmanifold)$ as to the minimal latent dimension of $\solmanifold$. Clearly, when $V\cong\mathbb{R}^{N_{h}}$ is finite-dimensional, the above definition is of interest only if $\mindim(\solmanifold)\ll N_{h}$. Nevertheless, as we will see below, this is always the case as soon as $\solmanifold$ has an intrinsic low-dimensional structure. 
Indeed, the value of $\mindim(\solmanifold)$ is strongly related to the topological properties of $\solmanifold$. For instance, in the case of compact sets, it is invariant under bicontinuous transformations. More precisely, we have the following.

\begin{theorem}
\label{theorem:homeom}
Let $V$ and $W$ be two Hilbert spaces. Let $\solmanifold\subset V$ and $\mathcal{M}\subset W$ be two compact subsets. If $\solmanifold$ and $\mathcal{M}$ are homeomorphic, then $\mindim(\solmanifold) = \mindim(\mathcal{M})$.
\end{theorem}

\proof{Since the situation is symmetric in $\mathcal{S}$ and $\mathcal{M}$, it is sufficient to prove that $\mindim(\mathcal{S})\ge\mindim(\mathcal{M})$. If $\mindim(\mathcal{S})=+\infty$, the inequality is obvious. Hence, we assume there exists some $\rdim\in\mathbb{N}$ for which $\delta_{\rdim}(\solmanifold)=0$. By definition of infimum, there exists a sequence of encoding-decoding pairs $\{(\encoder_{j}, \decoder_{j})\}_{j\ge0}$ in $\mathcal{C}(\solmanifold, \mathbb{R}^{\rdim})\times\mathcal{C}(\mathbb{R}^{\rdim}, V)$ such that $\sup_{u\in\solmanifold}\;||u-\decoder_{j}(\encoder_{j}(u))||_{V}\to0$ as $j\to+\infty$. Let now $\phi: \solmanifold\to\mathcal{M}$ be bicontinuous (recall that the sets are homeomorphic).
Since $\solmanifold$ is compact, $\phi$ admits a uniformly continuous extension $\tilde{\phi}:V\to W$ (cf. Theorem 1.12 in \cite{benyamini}).
We are then allowed to consider the continuous maps $\tildedecoder_j:=\tilde{\phi}\circ\decoder_{j}: \mathbb{R}^{\rdim}\to W$ and $\tildeencoder_j:=\encoder_{j}\circ\phi^{-1}: \mathcal{M}\to\mathbb{R}^{\rdim}$.
For $\omega$ a monotone modulus of continuity of $\tilde{\phi}$, we have

\begin{multline*}
    \delta_{\rdim}(\mathcal{M})\le\lim_{j\to+\infty}\;\sup_{m\in\mathcal{M}}||m-\tildedecoder_j(\tildeencoder_j(m))||_{W}=
    \\
    = \lim_{j\to+\infty}\;\sup_{m\in\mathcal{M}}||m-\tilde{\phi}(\decoder_{j}(\decoder_{j}'(\phi^{-1}(m))))||_{W}=
    \\
    = \lim_{j\to+\infty}\;\sup_{s\in\solmanifold}||\tilde{\phi}(s)-\tilde{\phi}(\decoder_{j}(\decoder_{j}'(s)))||_{W}\le
    \\
    \le \lim_{j\to+\infty}\;\sup_{s\in\solmanifold}\;\omega\left(||s-\decoder_{j}(\decoder_{j}'(s))||_{V}\right)\le
    \\
    \le \lim_{j\to+\infty}\;\omega\left(\sup_{s\in\solmanifold}||s-\decoder_{j}(\decoder_{j}'(s))||_{V}\right) = 0,
\end{multline*}
\\
as $\omega(h)\downarrow0$ whenever $h\downarrow0$. This proves that $\delta_{\rdim}(\solmanifold)=0\implies\delta_{\rdim}(\mathcal{M})=0$ and hence $\mindim(\mathcal{S})\ge\mindim(\mathcal{M})$. \qed
\newline
\newline}

\noindent The minimal latent dimension is also related to the so-called topological dimension, or Lebesgue covering dimension. For a formal definition of the latter we refer to \cite{drutu,engelking}. In particular, if $\solmanifold$ has an intrinsic $p$-dimensional structure, then we are able to bound $\mindim(\solmanifold)$ explicitly. Indeed, by classical results of Dimension Theory, the following theorem holds true.

\begin{theorem}
\label{theorem:mindimbounds}
Let $V$ be a Hilbert space and $\solmanifold\subset V$ a compact subset. If $\solmanifold$ has topological dimension $p$, then $\mindim(\solmanifold)\le 2p + 1$, and the infimum appearing in \eqref{eq:nonlinear-kolmogorov} is attained at all reduced dimensions $n\ge2p+1$.
Additionally, if $\solmanifold$ is a topological $p$-manifold, then the lower bound $\mindim(\solmanifold) \ge p$ also holds.
\end{theorem}

\proof{We shall prove that $\delta_{2p+1}(\solmanifold)=0$, and that the infimum is attained. To this end, we notice that $\solmanifold$ is compact and thus separable. Therefore, by the Menger–Nöbeling embedding theorem (see Theorem 1.11.4 in \cite{engelking}), there exists a subset $A\subset\mathbb{R}^{2p+1}$ and a bicontinuous map $\phi:\solmanifold\to A$. By continuity, the set $A$ is compact. In particular, $\phi^{-1}$ admits a continuous extension $\Psi: A\to V$. The existence of such an extension can be argued as in the proof Theorem \ref{theorem:homeom}, or using other results such as Dugundji extension theorem \cite{dugundji}, with the advantage of generalizing Theorem \ref{theorem:mindimbounds} to the case of normed spaces. Next, we define $\encoder:=\phi$. Then, the pair $(\encoder, \decoder)$ agrees with the definition of Nonlinear Kolmogorov $\rdim$-width and it also yields a perfect reconstruction of $\solmanifold$. The first claim in the theorem follows.

Assume now that $\solmanifold$ is a $p$-manifold and let $\rdim<p$. By definition, there exists a bicontinuous map $\phi$ from the closed unit ball $B:=\{x\in\mathbb{R}^{p},\;|x|\le1\}$ to a certain compact subset $U\subseteq{\solmanifold}$. Let    $$m:=\min_{|x|=1}\;||\phi(x)-\phi(-x)||.$$ Due to compactness, the minimum is attained and thus $m>0$ (recall that $\phi$ is bijective). We now prove that $\delta_{\rdim}(\solmanifold)\ge m/2$, and therefore $\mindim(\solmanifold)\ge p$.
Let $\encoder: \solmanifold\to\mathbb{R}^{\rdim}$ and $\decoder: \mathbb{R}^{\rdim}\to V$ be continuous. We note that the composition $\encoder\circ\phi$ is continuous from $B\subset\mathbb{R}^{p}\to\mathbb{R}^{\rdim}$. Therefore, as $\rdim<p$, by the Borsuk-Ulam theorem \cite{borsuk} we are granted the existence of a point $x^{*}\in B$, $|x^{*}|=1$, for which $\encoder(\phi(x^{*})) = \encoder(\phi(-x^{*}))=:z$. It follows that,
    \begin{multline*}
        \sup_{v\in\solmanifold}\;||v-\decoder(\encoder(v))||\ge
        \sup_{v\in U}\;||v-\decoder(\encoder(v))||
        =\sup_{x\in B}\;||\phi(x)-\decoder(\encoder(\phi(x)))||\ge
        \\
        \\
        \ge \max\;\{||\phi(x^{*})-\decoder(\encoder(\phi(x^{*})))||, 
        \;||\phi(-x^{*})-\decoder(\encoder(\phi(-x^{*})))||\}\ge
        \\
        \\
        \ge\frac{1}{2}||\phi(x^{*})-\decoder(z)||+\frac{1}{2}||\phi(-x^{*})-\decoder(z)||\ge\frac{1}{2}||\phi(x^{*})-\phi(-x^{*})||\ge\frac{m}{2},
    \end{multline*}
\newline
and the proof is complete.\qed
\newline
\newline}

We mention that, in the particular case of $p$-manifolds and under suitable smoothness assumptions, the bounds in Theorem \ref{theorem:mindimbounds} can be sharpened to $\mindim(\solmanifold)\le 2p$ or even $\mindim(\solmanifold)\le 2p-1$ in case $p$ is not a power of 2. These are all consequences of the so-called Whitney embedding theorem and a few of its variants. We do not dive deeper into the matter but leave \cite{skopenkov} as a reference for the interested reader. We also note that the intrinsic dimension of $\solmanifold$ does not uniquely determine the value of $\mindim(\solmanifold)$. In particular, $\solmanifold$ may have topological dimension $p$ but $\mindim(\solmanifold) > p$, coherently with Theorem \ref{theorem:mindimbounds}. In this respect, we report below two simple examples.

\begin{example}
\label{example:circle}
Let $V=\mathbb{R}^{2}$ and $\solmanifold=\{\x\in V : |\x|=1\}$ be the unit circle. Then, $\solmanifold$ is a one-dimensional manifold but $\delta_{1}(\solmanifold) = 1$ and $\mindim(\solmanifold)=2$. To see this, consider any pair of continuous maps $\encoder:\solmanifold\to\mathbb{R}$ and  $\decoder:\mathbb{R}\to\mathbb{R}^{2}$. By the Borsuk-Ulam theorem, there exists a point $\x\in\solmanifold$ such that $\encoder(\x)=\encoder(-\x)$. Therefore, being $||\cdot||=|\cdot|$ the Euclidean norm, 
\begin{multline*}
\sup_{\vb\in\solmanifold}\;||\vb-\decoder(\encoder(\vb))|| 
\ge
\max\{|\x-\decoder(\encoder(\x))|,\;|-\x-\decoder(\encoder(-\x))|\}
\ge
\\
\ge
\frac{1}{2}\left(|x-\decoder(\encoder(\x))|+|-\x-\decoder(\encoder(-\x))|\right)\ge\frac{1}{2}|\x-(-\x)|=1.
\end{multline*}
As $\encoder$ and $\decoder$ are arbitrary, we conclude that $\delta_{1}(\solmanifold)\ge1$. The equality is then obtained by considering the case in which both $\encoder$ and $\decoder$ are identically zero.
\end{example}

\begin{example}
\label{example:curve}
On the spatial domain $\Omega = (0,\pi)$, consider the boundary value problem
\[
\begin{cases}
u'' = -u & x\in\Omega
\\
u(0) = 10(2\mu^{3}-3\mu^{2}+\mu)&
\\
u'(\pi) = 1-2|1-2\mu|,&
\end{cases}
\]
where $\mu\in[0,1]$ is a parameter. Let us then consider the solution manifold
$\solmanifold=\{u_{\mu}\}_{\mu\in[0,1]}$ as a subset of $V=L^{2}(\Omega)$. Then, $\solmanifold$ is a 1-dimensional manifold but its minimal latent dimension equals $\mindim(\solmanifold)=2$. 
Indeed, explicitly expanding the solutions reads $$u_{\mu}(x) = 10(2\mu^{3}-3\mu^{2}+\mu)\cos x + (2|1-2\mu|-1)\sin x.$$ It is then clear that, up to scaling of the $L^{2}$-norm, $\solmanifold$ can be isometrically identified with the curve $\mu\to(20\mu^{3}-30\mu^{2}+10\mu,\;2|1-2\mu|-1)$ in $\mathbb{R}^{2}$. But the latter curve is a compact manifold with positive nonlinear Kolmogorov 1-width, as it is homeomorphic to the unit circle (see Figure \ref{fig:curve} and Theorem \ref{theorem:homeom}).
\end{example}

\noindent\textit{Remark.\;} Here we only considered the case of Hilbert spaces, which is the typical framework used for elliptic PDEs. However, as mentioned in the proof of Theorem \ref{theorem:mindimbounds}, many of the above ideas and results can be adapted to the more general context of normed and Banach spaces.

\begin{figure}
  \begin{minipage}[t]{0.38\textwidth}
    \includegraphics[width=\textwidth]{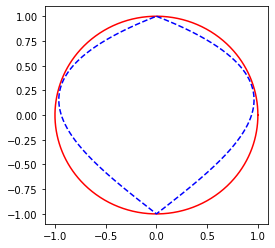}
  \end{minipage}\hfill
  \begin{minipage}[b]{0.62\textwidth}
    \caption{
    Reference picture for Example \ref{example:curve}, Section \ref{sec:nl-kw}. In blue (dashed line), the curve $\mu\to(20\mu^{3}-30\mu^{2}+10\mu,\;2|1-2\mu|-1)$, for $0\le\mu\le1$; in red (straight line), the unit circle. The two curves are clearly homeomorphic.
    \vspace{7em}
    } \label{fig:curve}
  \end{minipage}
\end{figure}

\subsection{Application to Parametrized PDEs}
\label{sec:nonlinear-param-pdes}
Let us now consider the case of a PDE that depends on a vector of $p$ parameters. We fix a parameter space $\Theta\subset\mathbb{R}^{p}$ and a Hilbert state space $V$. As before, for each $\mup\in\Theta$ we denote the corresponding PDE solution with $\up$. Similarly, we define $\solmanifold=\{\up\}_{\mup\in\Theta}$. Notice that we refer to $\solmanifold$ as the solution manifold even though, in fact, it is not granted that $\solmanifold$ is a manifold in the topological sense. This latter property can be recovered under additional hypotheses on the parameter space and the parametric map.

We consider the problem of finding a low-dimensional representation of $\solmanifold$ by means of nonlinear reduction. In particular, we wish to compress $\solmanifold$ as much as possible without paying in terms of accuracy, which corresponds to working with the minimal dimension $\mindim(\solmanifold)$. To this end, we must take into account the fact that the dimension of the parameter space $\Theta$ influences the low-dimensional structure of $\solmanifold$, in fact, the latter is ultimately defined in terms of $p$ scalar parameters. 

Parallel to this, one may also exploit the parameters as additional tools during the dimensionality reduction process. This corresponds to replacing the solution manifold with the augmented set $\solmanifold_{\Theta}:=\{(\mup, \up)\}_{\mup\in\Theta}\subset\mathbb{R}^{p}\times V$, where $\mup$ appears explicitly. The following Theorem provides some insights about both alternatives.

\begin{theorem}
\label{theorem:parametric-reduction}
Let $\mup\to\up$ be a map from a compact set $\Theta\subset\mathbb{R}^{p}$ to some Hilbert space $V$. Define the sets $\solmanifold:=\{\up\}_{\mup\in\Theta}$ and $\solmanifold_{\Theta}:=\{(\mup, \up)\}_{\mup\in\Theta}$. We have the following:
\begin{itemize}
        \item [a1)] if the map $\mup\to\up$ is Lipschitz continuous, then $\mindim(\solmanifold)\le 2p+1$.
        \item [a2)] if there exists at least an internal point $\mup_{0}\in\Theta$ where the correspondence $\mup\to\up$ is locally injective, then $\mindim(\solmanifold) \ge p$.
        \item [a3)] if the map $\mup\to\up$ is continuous and injective, then $\mindim(\solmanifold) = \mindim(\Theta)$. In particular, $\mindim(\solmanifold)=p$ whenever $\Theta$ has nonempty interior.
        \item [b1)] if the map $\mup\to\up$ is continuous, then $\mindim(\solmanifold_{\Theta})=\mindim(\Theta)$. In particular, $\mindim(\solmanifold_{\Theta})=p$ whenever $\Theta$ has nonempty interior.
\end{itemize}
\end{theorem}

\proof{For the sake of brevity, let us define the map $u:\Theta\to V$ as $u(\mup):=\up$.

\begin{itemize}
    \item [a1)] Let $\dim(\solmanifold)$ and $\dim_{H}(\solmanifold)$ be respectively the topological dimension (covering dimension) and the Hausdorff dimension of $\solmanifold$. A result due to Sznirelman, see Theorem 2.43 in \cite{drutu}, ensures that $\dim(\solmanifold)\le\dim_{H}(\solmanifold)$. Since $u$ is Lipschitz, we also have $\dim_{H}(\solmanifold) = \dim_{H}(u(\Theta))\le\dim_{H}(\Theta)\le p$, as $\Theta\subset\mathbb{R}^{p}$ and $\dim_{H}$ is known to be Lipschitz subinvariant. Thus $\dim(\solmanifold) \le p$ and the conclusion follows by Theorem \ref{theorem:mindimbounds} (notice that, since $\Theta$ is compact and $u$ is continuous, $\solmanifold$ is also compact).
    \\
    \item [a2)] Let $B\subseteq\Theta$ be a closed ball centered at $\mup_{0}$ such that $u_{|B}$, the restriction of $u$ to $B$, is injective. Then, $u_{|B}:B\to u(B)$ is a continuous bijection between compact metric spaces, which is enough to grant the existence and the continuity of the inverse map $u_{|B}^{-1}$. In particular, the sets $B$ and $u(B)$ are homeomorphic and, by Theorem \ref{theorem:homeom}, $p = \mindim(B) = \mindim(u (B))$. Since $\mindim(u(B))\le \mindim(\solmanifold)$, this proves (a2).
    \\
    \item [a3)] As in the proof of a2), we notice that $u$ admits a continuous inverse and is thus an homeomorphism between $\Theta$ and $\solmanifold$ (which are both compact). Then, $\mindim(\Theta)=\mindim(\solmanifold)$ by Theorem \ref{theorem:homeom}. Finally, if $\Theta$ has nonempty interior, we may select a closed ball $B\subseteq\Theta$ and notice that $p = \mindim(B) \le \mindim(\Theta)\le p\implies\mindim(\Theta)=p$.
    \\
    \item [b1)] Define the Hilbert space $\tilde{V}:=\mathbb{R}^{p}\times V$ and the map $U:\Theta\to\tilde{V}$ as $U(\mup):=(\mup, \up)$. Then $U$ is both continuous and injective. Thus, by a3) we have $\mindim(\Theta)=\mindim(U(\Theta))=\mindim(\solmanifold_{\Theta})$.
\end{itemize}\qed\newline}

\noindent
\noindent\textit{Remark.\;} Theorem \ref{theorem:parametric-reduction} holds for a generic Hilbert-valued map, meaning that the correspondence $\mup\to \up$ needs not to involve the solution of a PDE. Because of this generality, some hypotheses cannot be weakened. For instance, one cannot replace the requirement of Lipschitz continuity in statement (a1) with continuity, mainly because of space-filling curves. As a counterexample, consider the Hilbert space of real square summable sequences, $V=\ell^{2}$. Then, by a straightforward application of the Hahn–Mazurkiewicz theorem (see Theorem 3-30 in \cite{hocking}), there exists a continuous map from the unit interval to $\ell^{2}$ whose image $\solmanifold$ is the so-called Hilbert cube, informally $\solmanifold = \prod_{n=1}^{+\infty} [0, 1/n]$. Therefore, being $\Theta:=[0,1]$ the parameter space, we have a case in which $p=1$ but $\mindim(\solmanifold)=+\infty$. In fact, for each $\rdim\in\mathbb{N}$, the Hilbert cube contains an homeomorphic copy of the $\rdim$-dimensional unit cube $I^{\rdim}$. Thus, $\rdim = \mindim(I^{\rdim}) \le \mindim(\solmanifold)$ for all $n\ge0$ and so $\mindim(\solmanifold)=+\infty$.
\newline
\newline
Before moving to the actual description of our Deep Learning approach, we conclude this Section with a practical application of the results we have presented so far. In particular, we focus on the case of second order elliptic PDEs.

\subsubsection{Second Order Elliptic PDEs}

In order to state the main result, we first provide some notation. We denote by $\Omega$ a bounded domain in $\mathbb{R}^{d}$ and by $\cdot$ the scalar product in $\mathbb{R}^{d}$.
For $1\le q<+\infty$, we denote by $L^{q}(\Omega)$ the Lebesgue space of $q$-integrable real-valued maps; when $q=+\infty$, $L^{\infty}(\Omega)$ is defined as the Banach space of essentially-bounded maps. Similarly, we define the spaces $L^{q}(\Omega, \mathbb{R}^{d})$ and $L^{q}(\Omega,\mathbb{R}^{d\times d})$ in the Bochner sense, where $\mathbb{R}^{d}$ is considered with the Euclidean norm $|\cdot|$ and $\mathbb{R}^{d\times d}$ with the operator norm, $|\boldsymbol{A}|_{\mathbb{R}^{d\times d}}:=\sup_{|\xib|=1}|\boldsymbol{A}\xib|$. Given $k\in\mathbb{N}$, $1\le q<+\infty$, we write $W^{k,q}(\Omega)$ for the Sobolev space of all $w\in L^{q}(\Omega)$ that are $k$-times weakly differentiable with derivatives in $L^{q}(\Omega)$. We use $W_{0}^{k,q}(\Omega)$ to denote the subspace of all $w\in W^{k,q}(\Omega)$ that vanish on $\partial\Omega$, and we write $W^{-k,q}(\Omega)$ for the dual space of $W^{k,q}_{0}(\Omega)$ with respect to the duality product $\langle f,g\rangle\to\int_{\Omega}fg$. In order to prescribe Dirichlet boundary data, we also make use of the Sobolev-Slobodeckij spaces $W^{s,q}(\partial\Omega)$, where $s>0$ is typically not an integer. All the aforementioned spaces are considered with their usual norms, see e.g. \cite{evans}.
\newline
\newline
We define the sets of all admissible conductivity tensor-fields and transport fields, respectively $\Sigma(\Omega)\subset L^{\infty}(\Omega, \mathbb{R}^{d\times d})$ and $B(\Omega)\subset L^{\infty}(\Omega, \mathbb{R}^{d\times d})$, as follows. We let $\cond\in\Sigma(\Omega)$ if and only if it is uniformly elliptic, that is, there exists $\varepsilon>0$ such that for almost all $x\in\Omega$ one has $\cond(x)\xib\cdot \xib\ge\varepsilon|\xib|^{2}$ for all $\xib\in\mathbb{R}^{d}$. We let $\adve\in B(\Omega)$ if and only if it is differentiable and divergence free, that is, $\adve\in\mathcal{C}^{1}(\Omega, \mathbb{R}^{d})$ and div$(\adve)=0$ in $\Omega$. We endow both $\Sigma(\Omega)$ and $B(\Omega)$ with the infinity norm $||\cdot||_{\infty}$. We are now able to state the following.

\begin{theorem}
\label{theorem:advediff-reduction}
Let $\Omega\subset\mathbb{R}^{d}$ be a bounded domain with Lipschitz boundary, and let $\Theta\subset\mathbb{R}^{p}$ be a compact subset with nonempty interior. Let $q\ge 2d/(d+2)$ be finite, and define the conjugate exponent $q':=q/(q-1)$. Moreover, let $\mup\to\condp\in\Sigma(\Omega),$ $\mup\to\advep\in B(\Omega),$ $\mup\to\forcep\in W^{-1,q'}(\Omega)$ be parameter dependent coefficients and $\mup\to \gp\in W^{1/q',q}(\partial\Omega)$ boundary data. For each $\mup\in\Theta$, we define $\up\in W^{1,q}(\Omega)$ as the unique solution to the following second order elliptic PDE

\[
\begin{multlined}
u\in W^{1,q}(\Omega):\\ u_{|\partial\Omega}=\gp\text{ and }\int_{\Omega}\condp\nabla u\cdot\nabla w + \int_{\Omega}\left(\advep\cdot\nabla u\right) w = \int_{\Omega}\forcep w\quad\forall w\in W^{1,q'}_{0}(\Omega),
\end{multlined}\]
\newline
Consider the solution manifold $\solmanifold:=\{\up\}_{\mup\in\Theta}$ as a subset of $V:=L^{2}(\Omega)$. The following hold true:
\begin{itemize}
    \item [i)] if the dependence of $\condp$, $\advep$, $\forcep$, $\gp$ on $\mup$ is Lipschitz continuous, then $\mindim(\solmanifold)\le 2p+1$.
    \item [ii)] if $\condp$, $\advep$, $\forcep$, $\gp$ depend continuously on $\mup$ and the solution map $\mup\to \up$ is one-to-one, then $\mindim(\solmanifold)=p$.
\end{itemize}
Additionally, let $\solmanifold_{\Theta}:=\{(\mup, \up)\}_{\mup\in\Theta}\subset\mathbb{R}^{d}\times V$ be the augmented manifold. Then:
\begin{itemize}
    \item [iii)] if $\condp$, $\advep$, $\forcep$, $\gp$ depend continuously on $\mup$, then $\mindim(\solmanifold_{\Theta})=p$.
\end{itemize}
\end{theorem}

\proof{First of all, we notice that if the data $\condp,\advep,f_{\mup},g_{\mup}$ depend continuously on $\mup$, then so does $\up\in W^{1,q}(\Omega)$. This is easily proven by composition (cf. Lemma \ref{lemma:pde-continuity} in the Appendix). Also, the compactness of $\Theta$ implies that of the subsets
        $$\{\condp\}_{\mup\in\Theta}\subset\Sigma(\Omega),\quad\{\advep\}_{\mup\in\Theta}\subset B(\Omega),$$
        $$\quad\{\forcep\}_{\mup\in\Theta}\subset W^{-1,q'}(\Omega),\quad\{g_{\mup}\}_{\mup\in\Theta}\subset W^{1/q',q}(\partial\Omega).$$
Therefore, whenever the coefficients and the boundary data are Lipschitz continuous in $\mup$, so is the solution map $\Theta\to W^{1,q}(\Omega)$ (by composition, cf. Lemma \ref{lemma:pde-continuity}).
Finally, since $q\ge 2d/(d+2)$, we have the embedding $W^{1,q}(\Omega)\hookrightarrow L^{2}(\Omega)$ according to classical Sobolev inequalities (cf. Theorem 5.4 in \cite{adams}). In particular, all the aforementioned properties are preserved if we consider the parametric map as taking values in $V:=L^{2}(\Omega)$.
Statements (i), (ii) and (iii) now directly follow from Theorem \ref{theorem:parametric-reduction}.\qed\newline\newline}

\noindent\textit{Remark.\;} In Theorem \ref{theorem:advediff-reduction}, the PDE is firstly solved in the Banach space $W^{1,q}(\Omega)$ and the solution manifold is then embedded in the Hilbert space $L^{2}(\Omega)$. This construction allows for a large spectrum of PDEs where the solution $\up$ may exhibit singularities.
A remarkable example is found for the dimensions $d=2,3$, where singular forces such as Dirac delta functions produce solutions $\up\notin H^{1}(\Omega):=W^{1,2}(\Omega)$ \cite{casas}. In these cases, the above Theorem still applies, e.g. with $q'=4$ and $q=4/3$ (cf. Morrey embedding). 
Nevertheless, we shall point out that in the Hilbert case, $q'=q=2$, it is possible to restrict the state space to $V=H^{1}(\Omega)\subset L^{2}(\Omega)$. Note also that in this case the condition $q\ge 2d/(d+2)$ is redundant as it is trivially satisfied for any $d\ge1$, coherently with the fact that $H^{1}(\Omega)$ always embeds in $L^{2}(\Omega)$.

\section{Learning the Solution Manifold by means of Neural Networks}
\label{sec:learning}

\review{We now present our Deep-Learning approach to Reduced Order Modelling (DL-ROM). After a brief recap on deep feedforward neural networks, Section \ref{sec:nn}, we move to the actual description of the DL-ROM, where we discuss both its theoretical and numerical properties (Section \ref{sec:dlrom}). Then, respectively in Sections \ref{sec:ae} and \ref{sec:cc}, we dive deeper into the design choices for the nonlinear dimensionality reduction and the approximation of the reduced map.}

\subsection{Neural Networks}
\label{sec:nn}
Neural networks are nonlinear universal approximators that have recently gained a lot of popularity in several fields such as Machine Learning, Statistics and Approximation Theory. The fundamental building block of a neural network is the \textit{layer}. Given two state spaces $V_{1}:=\mathbb{R}^{n_{1}}$, $V_{2}:=\mathbb{R}^{n_{2}}$ and a scalar-valued function $\rho:\mathbb{R}\to\mathbb{R}$, a layer $L$ with activation $\rho$ is a map $L:V_{1}\to V_{2}$ of the form $$L(\vb) = \rho\left(\weight \vb+\bias\right),$$ where $\weight\in\mathbb{R}^{n_{2}\times n_{1}}$ and $\bias\in\mathbb{R}^{n_{2}}$ are respectively the weights and biases of the layer. Notice that, as $\weight \vb+\bias$ is $n_{2}$-dimensional, the operation $\weight \vb+\bias\to\rho\left(\weight \vb+\bias\right)$ is intended componentwise.

Deep neural networks are defined through the composition of multiple layers. More precisely, a DNN with $l\ge0$ \textit{hidden layers} is a map of the form $$\Phi=L_{l+1}\circ L_{l}...\circ L_{1},$$ where each $L_{i}$ is a layer. The layers $L_{i}: \mathbb{R}^{n_{i}}\to\mathbb{R}^{n_{i+1}}$, for $i=1,..,l,$ are called hidden layers, while $L_{l+1}$ is the output layer. The latter is sometimes assumed to have no activation, as in \cite{kutyniok,schwab}. We refer to $l$ as to the depth of the network. When $l=1$, the adjective \textit{deep} is usually dropped and $\Phi$ is said to be \textit{shallow}. Note that we also allow for the degenerate case $l=0$, where the NN actually reduces to the output layer. This is somewhat unusual, but it will help us in making the notation lighter. Finally, we say that $\Phi$ has activation function $\rho$, or equivalently that $\Phi$ is a $\rho$-DNN, if all of its (hidden) layers share that same activation.
\newline
\newline
The practical implementation of a neural networks is usually done as follows. At first, one designs the NN architecture. This corresponds to choosing (i) the depth $l$, (ii) the number of neurons within each layer $L_{i}$, i.e. the output dimension $n_{i+1}$, and (iii) suitable constraints on the weights and biases. Among other things, the latter is what makes the difference between dense layers (no constraints on $\weight$) and convolutional layers ($\weight$ is sparse and multiple values are shared), which heavily impacts on the network \textit{complexity} \cite{gu}. The second step regards the network training. There, the weights and biases are tuned and optimized according to a suitable loss function. The optimization is typically performed using gradient descent algorithms \cite{schmidhuber}. To measure the network complexity, we count its degrees of freedom, that is the number of scalar parameters that are actually optimized during the training. Therefore, the contribution of each layer $L_{i}$ depends on the corresponding type. For dense layers it is $(n_{i}+1)n_{i+1}$, while it can be substantially lower in sparse and convolutional layers.

\subsection{Deep-Learning based Reduced Order Model}
\label{sec:dlrom}
\begin{figure}
\includegraphics[width=270pt]{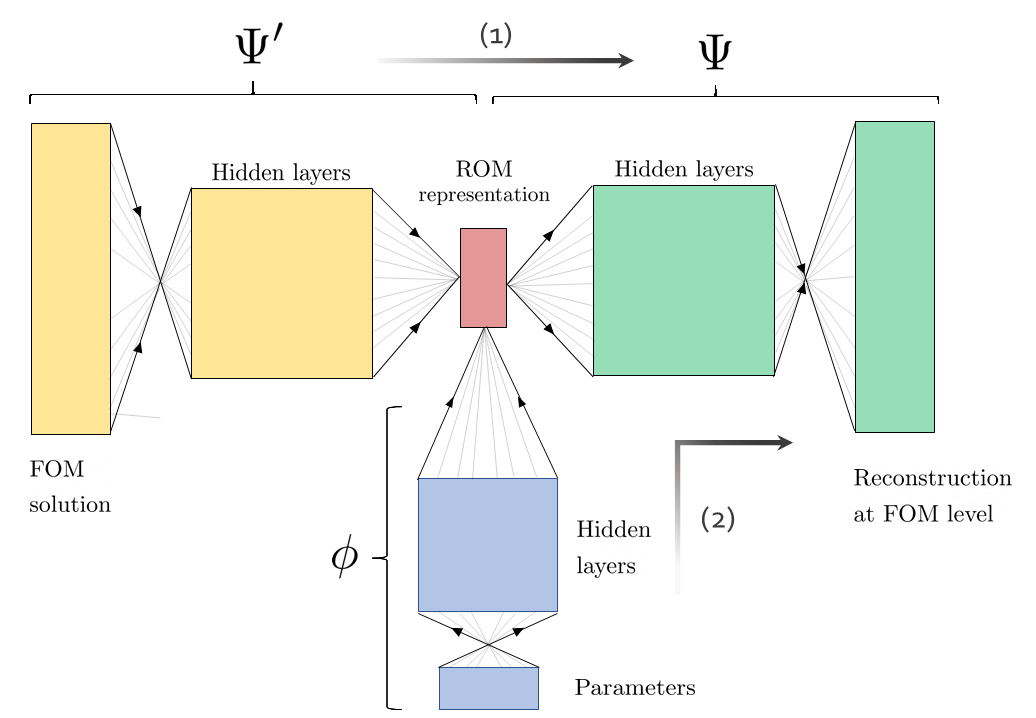}
\caption{
    \footnotesize{Workflow in the DL-ROM approach. The whole process consists of three neural networks, $\encoder$, $\decoder$ and $\phi$. First, the autoencoder $\decoder\circ\encoder$ is trained over several high-fidelity snapshots in order to learn an approximation of the identity operator over the solution manifold (step 1). In this way, the encoder provides a low-dimensional representation of the FOM solutions, from $\ufomp\in\mathbb{R}^{\fomdim}$ to $\uromp:=\encoder(\ufomp)\in\mathbb{R}^{n}$. At this point, the third network, $\phi$, is trained to learn the map $\mup\to\uromp$ (step 2). 
    Finally, the composition $\Phi:=\decoder\circ\phi$ defines the DL-ROM approximation of the parameter-to-state map.}
}
\label{fig:workflow}
\end{figure}
We are given a parameter space $\Theta\subset\mathbb{R}^{p}$, a parameter dependent PDE and a high-fidelity FOM $\mup\to \ufomp\in\mathbb{R}^{N_{h}}$. 
Our purpose is to approximate the solution map by means of a suitable neural network $\Phi:\mathbb{R}^{p}\to\mathbb{R}^{N_{h}}$. 
\review{For the sake of simplicity, through the whole section, we make the following assumption.
\begin{assump}
\label{assumption:dl-rom}
All DNNs use the same activation function $\rho:\mathbb{R}\to\mathbb{R}$ for the hidden layers,  where $\rho$ is Lipschitz continuous and not a polynomial. The parameter space $\Theta$ is compact and the parametric map $\mup\to\textnormal{\textbf{u}}_{\mup}^{h}$ is continuous.
\end{assump}
\noindent A typical activation function satisfying the above requirements is the so-called $\alpha$-leaky ReLU, i.e. $\rho(x) = x\boldsymbol{1}_{[0,+\infty)}(x) + \alpha x\boldsymbol{1}_{(-\infty, 0)}(x)$ where $\alpha>0$ is fixed. In order to build $\Phi$, we mimic the two steps paradigm of the Reduced Basis method, yielding the workflow depicted in Figure \ref{fig:workflow}. This corresponds to introducing the three networks below,
$$\Psi':\mathbb{R}^{\fomdim}\to\mathbb{R}^{n},\quad\quad\Psi:\mathbb{R}^{\rdim}\to\mathbb{R}^{\fomdim}$$
$$\phi:\mathbb{R}^{p}\to\mathbb{R}^{\rdim}.$$
The first two, respectively the encoder $\encoder$ and the decoder $\decoder$, serve for the nonlinear dimensionality reduction of the solution manifold $\solmanifoldh:=\{\ufomp\}_{\mup\in\Theta}\subset\mathbb{R}^{N_{h}}$, which we map onto $\mathbb{R}^{n}$. According to our previous analysis, we set the latent dimension to be $n:=\mindim(\solmanifoldh)$. As discussed in Theorem \ref{theorem:parametric-reduction}, this often translates to $n\le 2p+1$, resulting in a massive reduction whenever $p\ll N_{h}$. The purpose of the third network is to approximate the reduced parametric map $\mathbb{R}^{p}\ni\mup\to\Psi'(\ufomp)\in\mathbb{R}^{n}$, so that the final ROM is obtained by composition of $\phi$ and $\Psi.$ 

At the very end, the role of the encoder $\Psi'$ is only auxiliary as the DL-ROM ultimately results in a single network $\Phi:=\decoder\circ\phi$. However, we believe that our construction significantly facilitates the practical problem of designing the architectures. This is because the three networks have very different purposes. $\Psi'$ and $\Psi$ are required to learn the intrinsic characteristics of the solutions, so their complexity is related to the richness of the solution manifold and the geometrical properties of the solutions. Conversely, $\phi$ needs to understand the interplay between solutions and parameters, which can result in a very complicated relation even if the solution manifold is fairly simple (e.g. linear). As the design of DNN architectures is still far from obvious, we believe that this perspective can be of help in practical implementations. Nonetheless, this splitting of the ROM also allows for a few considerations on the numerical errors, that we discuss below.\\\\\noindent We notice that, due to assumption \eqref{assumption:dl-rom}, all the networks in the DL-ROM pipeline are Lipschitz continuous. Also, without loss of generality, we can assume that $\Psi$ has a Lipschitz constant equal to 1. In fact, for any $C>0$, the maps $\tilde{\Psi}'(\x):=C\Psi'(\x)$ and $\tilde{\Psi}(\x):=\Psi(\x/C)$ define the same autoencoder as $\Psi'$ and $\Psi$. As a consequence, the worst-case approximation error of the DL-ROM
$$\mathcal{E}_{A}:=\sup_{\mu\in\Theta}||\ufomp-\Psi(\phi(\mup))||,$$
can be bounded as $\mathcal{E}_{A}\le\mathcal{E}_{R}+\mathcal{E}_{P}$, the latter being respectively the reconstruction error and the parametric error,
$$\mathcal{E}_{R}:=\sup_{\ub^{h}\in\solmanifoldh}||\ub^{h}-\Psi(\Psi'(\ub^{h}))||,\quad\quad
\mathcal{E}_{P}:=\sup_{\mu\in\Theta}|\Psi'(\ufomp)-\phi(\mup)|.$$
where $|\cdot|$ is the Euclidean norm, while we recall that $||\cdot||$ comes from the metric originally chosen over the state space $V_{h}\subset V$. In fact,
\begin{multline*}
    \sup_{\mu\in\Theta}||\ufomp-\Psi(\phi(\mup))||\le \sup_{\mu\in\Theta}||\ufomp-\Psi(\Psi'(\ufomp))||+\sup_{\mu\in\Theta}||\Psi(\Psi'(\ufomp))-\Psi(\phi(\mup))||
    \le\\
    \le \sup_{\ub^{h}\in\solmanifoldh}||\ub^{h}-\Psi(\Psi'(\ub^{h}))||+\sup_{\mu\in\Theta}|\Psi'(\ufomp)-\phi(\mup)|.
\end{multline*}
\final{We remark that both $\mathcal{E}_{R}$ and $\mathcal{E}_{P}$ can be made arbitrarily small. In fact, as proven by Pinkus back in 1999 \cite{pinkus}, DNNs are dense in the space of continuous functions defined over compact domains (note that here our assumption on $\rho$ is crucial). Therefore, since $\Theta$ is compact and $\mup\to\Psi'(\ufomp)$ is continuous, the parametric error can become as small as possible. Similarly, the reconstruction error can get closer and closer to the limit value $\delta_{n}(\solmanifoldh)=0$. In fact, we can approximate the reconstruction provided by any two continuous maps $\Psi_{*}':\solmanifoldh\to\mathbb{R}^{\rdim}$ and $\Psi_{*}:\mathbb{R}^{\rdim}\to\mathbb{R}^{\fomdim}$ using DNNs. To see this, fix any $\varepsilon>0$ and let $\mathcal{V}:=\Psi_{*}(\solmanifoldh)$ be the embedded solution manifold. Since $\mathcal{V}$ is compact, the aforementioned density result ensures the existence of some DNN $\Psi$ that approximates $\Psi_{*}$ over $\mathcal{V}$ upto an error of $\varepsilon$. Similarly, there exists a DNN $\Psi'$ that approximates $\Psi'_{*}$ over $\solmanifoldh$ upto an error of $\varepsilon/C$, where $C>0$ is the Lipschitz constant of $\Psi$. Then, for any $\ub^{h}\in\solmanifoldh$ one has \begin{multline*}
||\ub^{h}-\Psi(\Psi'(\ub^{h}))||\le\\\le||\ub^{h}-\Psi_{*}(\Psi'_{*}(\ub^{h}))||+||\Psi_{*}(\Psi'_{*}(\ub^{h}))-\Psi(\Psi'_{*}(\ub^{h}))||+||\Psi(\Psi'_{*}(\ub^{h}))-\Psi(\Psi'(\ub^{h}))||\le\\
\le ||\ub^{h}-\Psi_{*}(\Psi'_{*}(\ub^{h}))|| + \varepsilon + C|\Psi'_{*}(\ub^{h}))-\Psi'(\ub^{h})|\le\\\le||\ub^{h}-\Psi_{*}(\Psi'_{*}(\ub^{h}))||+2\varepsilon.
\end{multline*}
This shows that $\mathcal{E}_{R}$ can reach the limit value $\delta_{n}(\solmanifoldh)$. In particular, thanks to our design choice of letting $n=\mindim(\solmanifoldh)$, the reconstruction error can get arbitrarily close to zero.

In general, all the above reasoning suggests a two step approach where we first train the autoencoder $\Psi\circ\Psi'$ and then the reduced map $\phi$.} Nevertheless, before studying the two steps of the DL-ROM separately, some analysis of the networks complexity is needed. In fact, while the DL-ROM can reach any level of accuracy, the size of the networks involved may grow quickly, making their optimization problematic. The result below provides a \final{first} answer to such question.

\begin{theorem} Under Assumption \ref{assumption:dl-rom}, let $\rho$ be the ReLU activation function. Assume that the map 
$\mup\to\textnormal{\textbf{u}}_{\mup}^{h}$ is Lipschitz continuous for some constant $L>0$, and that the infimum in $\eqref{eq:nonlinear-kolmogorov}$ is attained, i.e. 
there exists two continuous maps $\Psi_{*}':\mathbb{R}^{\fomdim}\to\mathbb{R}^{n}$ and $\Psi_{*}:\mathbb{R}^{n}\to\mathbb{R}^{\fomdim}$ such that
$$\Psi_{*}(\Psi'_{*}(\textnormal{\textbf{u}}))=\textnormal{\textbf{u}}\quad\quad\forall\textnormal{\textbf{u}}\in\solmanifoldh.$$
Additionally, assume that $\Psi_{*}'$ and $\Psi_{*}$ are $s$-times differentiable, $s\ge2$, and have bounded derivatives. Let
$$C_{1}=\sup_{|\boldsymbol{\alpha}|\le s}\sup_{\textnormal{\textbf{u}}\in\mathbb{R}^{\fomdim}}|D^{\boldsymbol{\alpha}}\Psi'_{*}(\textnormal{\textbf{u}})|,
\quad\quad C_{2}=\sup_{|\boldsymbol{\alpha}|\le s}\sup_{\nub\in\mathbb{R}^{n}}||D^{\boldsymbol{\alpha}}\Psi_{*}(\nub)||.$$
For any $0<\varepsilon<1$, let $m\in\mathbb{N}$ be the first integer for which $d_{m}(\solmanifoldh)<\varepsilon$. Then, for some constant $c=c(\Theta, L,C_{1},C_{2},p,n,s)$, there exists a DL-ROM with a decoder $\Psi$ having at most
\begin{itemize}
    \item [i)] $cm^{1+n/(s-1)}\varepsilon^{-n/(s-1)}\log(m/\varepsilon)+m\fomdim$ active weights
    \item [ii)] $c\log(m/\varepsilon)$ layers
\end{itemize}
and a reduced map $\phi$ having at most
\begin{itemize}
    \item [iii)] $c\varepsilon^{-p}\log(1/\varepsilon)$ active weights
    \item [iv)] $c\log(1/\varepsilon)$ layers
\end{itemize}
such that the approximation error satisfies $\mathcal{E}_{A}<2\varepsilon$.
\end{theorem}
\begin{proof}
Let $\Nu:=\Psi'_{*}(\solmanifoldh)\subset\mathbb{R}^{n}$ be the embedded solution manifold. The latter is a compact subset of diameter at most $\diam(\Nu)\le LC_{1}\diam(\Theta)$. Let $m$ be as in the Theorem. 
Then there exists an orthonormal matrix $\mathbf{V}\in\mathbb{R}^{\fomdim\times m}$ such that $||\ub-\mathbf{V}\mathbf{V}^{T}\ub||<\varepsilon$ for all $\ub\in\solmanifoldh$. Define $F:\mathbb{R}^{n}\to\mathbb{R}^{m}$ as $F(\nub):=\mathbf{V}^{T}\Psi_{*}(\nub)$. Then $F$ is $s$-times differentiable and
$$\sup_{\nub\in\Nu}|D^{\boldsymbol{\alpha}}F(\nub)|\le\sup_{\nub\in\Nu}||D^{\boldsymbol{\alpha}}\Psi_{*}(\nub)||\le C_{2},$$
for any multi-index $\boldsymbol{\alpha}$ with $0\le|\boldsymbol{\alpha}|\le s$. In particular, as a direct consequence of Theorem 4.1 in \cite{guhring}, there exists a ReLU DNN $\psi:\mathbb{R}^{n}\to\mathbb{R}^{m}$ of depth at most $C\log(m/\varepsilon)$ and active weights at most 
$mC(m\varepsilon)^{-n/(s-1)}\log(m/\varepsilon)$ such that for all $j=1,\dots,m$ one has
$$\sup_{\nub\in\Nu}|F_{j}(\nub)-\psi_{j}(\nub)|\le\frac{\varepsilon}{2m},\quad\quad \esssup_{\nub,\;\nub'\in\Nu}\;\frac{|(F_{j}-\psi_{j})(\nub)-(F_{j}-\psi_{j})(\nub')|}{|\nub-\nub'|}\le\frac{\varepsilon}{2m}$$
where $C>0$ is a constant depending on $C_{2},n,s$ and $\diam(\Nu)$ (thus on $\Theta,L$ and $C_{1}$), wherease $F_{j}$ and $\psi_{j}$ are the $j$th components of the two vector-valued maps. In particular, we also have a control on the Lipschitz constant of $\psi$, which we can bound by $C_{2}+\varepsilon/(2m)\le C_{2}+1$. We now define the decoder DNN $\Psi(\nub):=\mathbf{V}\psi(\nub)$ as a network with no activation on the output layer, so that for any $\nub\in\Nu$ we have
\begin{multline*}
    ||\Psi_{*}(\nub)-\Psi(\nub)||\le||\Psi_{*}(\nub)-\mathbf{V}\mathbf{V}^{T}\Psi_{*}(\nub)||+||\mathbf{V}\mathbf{V}^{T}\Psi_{*}(\nub)-\Psi(\nub)||\le\\\le\varepsilon + |F(\nub)-\psi(\nub)|\le\frac{3}{2}\varepsilon
\end{multline*}
as $\Psi_{*}(\nub)\in\solmanifoldh$. Clearly, $\Psi$ has the same depth of $\psi$ up to one layer, and it comes with $mN_{h}$ additional weights. Define now $\phi_{*}:\mathbb{R}^{p}\to\mathbb{R}^{n}$ as $\phi_{*}(\mup):=\Psi'_{*}(\ufomp)$. Then $\phi_{*}$ is Lipschitz continuous with constant at most equal to $LC_{1}$ and is bounded in norm by $C_{1}$. Thus, we can apply again Theorem 4.1 in \cite{guhring}, this time with respect to the infinity norm, to obtain a DNN $\phi:\mathbb{R}^{p}\to\mathbb{R}^{n}$ such that for all $\mup\in\Theta$
$$|\phi_{*,j}(\mup)-\phi_{j}(\mup)|<\frac{\varepsilon}{2n(C_{2}+1)}\quad\quad\forall j=1,\dots,n$$
where $\phi$ has at most $n\tilde{C}(nC_{2}\varepsilon)^{-p}\log(nC_{2}/\varepsilon)$ weights and at most $\tilde{C}\log(nC_{2}/\varepsilon)$ layers, $\tilde{C}$ being a constant only dependent on $\Theta,L$ and $C_{1}$. In particular, if we let $\tilde{C}$ absorb the dependence with respect to $n$ and $C_{2}$, and we set $c:=\max\{C,\tilde{C}\}$, then the architectures of $\Psi$ and $\phi$ satisfy the claimed requirements. Finally, we note that $\ufomp=\Psi_{*}(\Psi'_{*}(\ufomp))=\Psi_{*}(\phi_{*}(\mup))$, due to our hypothesis on the perfect embedding. Therefore, the approximation error for the DL-ROM with decoder $\Psi$ and reduced map $\phi$ is bounded by
\begin{multline*}
 \mathcal{E}_{A}\le\sup_{\mup\in\Theta}||\Psi_{*}(\phi_{*}(\mup))-\Psi(\phi_{*}(\mup))||+\sup_{\mup\in\Theta}||\Psi(\phi_{*}(\mup))-\Psi(\phi(\mup))||\le\\\le 
 \sup_{\nub\in\Nu}||\Psi_{*}(\nub)-\Psi(\nub)||+(C_{2}+1)\sup_{\mup\in\Theta}|\phi_{*}(\mup)-\phi(\mup)|\le\\
 \le \frac{3}{2}\varepsilon + (C_{2}+1)\frac{\varepsilon}{2(C_{2}+1)}=2\varepsilon.
\end{multline*}
\end{proof}

\noindent Theorem 5 suggests that the DL-ROM approach can take advantage of intrinsic regularities in the solution manifold, even if the parameter-to-solution map is just Lipschitz continuous. This situation reflects the case in which although the solutions depend in a complicated way with respect to the parameters, the solution operator has good analytical properties. For instance, it is known that the solution operator of elliptic PDEs is analytic with respect to the coefficients \cite{babuska,hoffmann}. Thus, we can think of $\mup\to\phi(\mup)$ as a change of coordinates that enables a smooth description of the solutions. 

Secondly, we note that an important role is played by the parameter $m$. This is in agreement with other results in the literature, see e.g. Theorem 4.3 in \cite{kutyniok}, and it suggests a link between the DL-ROM complexity and the linear Kolmogorov $m$-width. We may interpret $m$ as an \textit{equivalent linear dimension}: in fact, the DL-ROM accuracy in Theorem 5 is roughly equivalent to the optimal one achievable via projections on $m$-dimensional subspaces. In this sense, we can think of $m$ as being the number of modes in a Reduced Basis approach or, analogously, the number of trunk nets in a DeepONet based ROM \cite{karniadakis,mishra2}. In the case of DL-ROMs, the value of $m$ does not affect the latent dimension but has an impact on the DNNs complexity: the slower $d_{m}(\solmanifoldh)$ decays, the more degrees of freedom in the DNN architecture and, consequently, the higher the number of training snapshots required for the optimization. Conversely, if the linear width decays mildly, then a mix of linear and nonlinear reduction may be an interesting choice, as in the recently proposed POD-DL-ROM approach \cite{fresca}.

Nevertheless, we mention that the complexity bounds for the decoder are suboptimal in the way they include the FOM dimension $\fomdim$. In fact, the extra contribute $mN_{h}$ comes from the choice of considering the state space $V_{h}$ as consisting of vectors rather than functions. In particular, we expect that better estimates can be found if the solutions are smooth with respect to the space variable $\x\in\Omega$. This goes in favor of architectures that explicitly account for space dependency, such as convolutional layers, or even mesh-free approaches, such as DeepONets.}

\subsubsection{Dimensionality reduction}
\label{sec:ae}
\review{We propose two alternative ways for compressing the solution manifold. The first one is completely unsupervised, in the sense that it only operates on the solutions irrespectively of the parameter values, and it is based on the use of autoencoders. The second one is a variation of the previous where we explicitly include $\mup$ in the encoding process. We detail them below.
\\
\paragraph{\textit{Autoencoder approach}} According to the reasoning in Section \ref{sec:dlrom}, we let $n:=\mindim(\solmanifoldh)$ and we introduce two DNN architectures, an encoder $\Psi':\mathbb{R}^{\fomdim}\to\mathbb{R}^{n}$ and a decoder $\Psi:\mathbb{R}^{n}\to\mathbb{R}^{\fomdim}$, which we design as follows. In principle, the encoder can be very simple, as its only purpose is to provide a different representation for each solution. The hard job is left to the decoder that needs to perform the reconstruction. In this sense, a plain design choice can be $\Psi'(\ub):=\rho\left(\weight\ub+\bias\right)$, i.e. to use a degenerate architecture with no hidden layers. Conversely, designing the decoder requires a little extra caution. If $\Omega$ is an hypercube, a good choice is to employ dense layers at the beginning and conclude with a block of convolutional layers as done in \cite{fresca,lee}. This allows the decoder to account for spatial correlations and be sufficiently expressive without growing too much in complexity. Indeed, convolutional layers have been proven to be very effective in image reconstruction tasks, and we see a clear analogy with our setting when $\Omega$ is an hypercube. More complicated geometries may require different strategies, but the terminal part of the decoder should still consist of sparse layers of some sort (such as those in Graph Convolutional Networks \cite{scarselli}). The expressiveness of the decoder may be increased in several ways. Empirically, we see that interesting results can be obtained for fixed depths but varying number of \textit{channels} in the convolutional layers.

Once the architecture has been fixed, we optimize the autoencoder by minimizing the loss function below
$$\lossf(\encoder,\decoder)=\frac{1}{\ntrain}\sum_{i=1}^{\ntrain}\loss(\ui,\;\decoder(\encoder(\ui))),$$
where $\loss$ is a suitable measure of discrepancy. A classical choice is to consider squared errors, $\loss(\y,\hat{\y})=||\y-\hat{\y}||^{2}$, in order to favor differentiability of the loss function. However, other metrics, such as relative errors $\loss(\y,\hat{\y})=||\y-\hat{\y}||/||\y||$, can be used as well. The minimization of the loss function is handled via stochastic gradient descent, mainly using batching strategies and first order optimizers.
\\
\paragraph{\textit{Transcoder-decoder approach}} As an alternative, we also propose a different architecture where the encoder is replaced with a \textit{transcoder} $\transcoder:\mathbb{R}^{p}\times\mathbb{R}^{\fomdim}\to\mathbb{R}^{n}$. The idea is to facilitate the encoding by making explicitly use of the parameters, so that different solutions are more likely to have different latent representations. This is clearly linked with Theorem \ref{theorem:parametric-reduction}.b1, and has the advantage of always enabling a maximal reduction, as we can now set $n=p=\mindim(\{\mup,\ufomp\}_{\mup\in\Theta})$. We define the decoder exactly as before, so that $\ufomp\approx\Psi(\transcoder(\mup,\ufomp))$. We refer to the combined architecture, $\Psi\circ\transcoder$, as to a \textit{transcoder-decoder}. In practice, the transcoder-decoder is analogous to an autoencoder but has $p$ additional neurons in the input layer, which is where we pass the parameters. To design the architectures, we follow the same rule of thumb as before. In general, we give more weight to the decoder, where we employ deep convolutional networks, while we use lighter architectures for the transcoder. For instance, in the limit case of 0-depth, the latter becomes of the form $\transcoder(\mup,\ub) = \rho\left(\weight'\mup + \weight\ub + \bias\right)$. During the offline stage, the transcoder-decoder is trained over the snapshots by minimizing the loss function below,
$$\lossf(\transcoder,\decoder)=\frac{1}{\ntrain}\sum_{i=1}^{\ntrain}\loss(\ui,\;\decoder(\transcoder(\mui,\ui))),$$
where $\loss$ is as before. The two approaches, autoencoder and transcoder-decoder, adopt different perspectives and provide different advantages. The first one is completely based on the solution manifold, so it is likely to reflect intrinsic properties of $\solmanifoldh$. On the other hand, the transcoder-decoder ensures a maximal compression, the latent dimension being always equal to $p$. In particular, the latent coordinates can be seen as an alternative parametrization of the solution manifold. In this sense, we say that $\transcoder$ performs a transcoding.}

\subsubsection{Approximation of the reduced map}
\label{sec:cc}
The second step in the DL-ROM pipeline is to approximate the reduced map $\mathbb{R}^{p}\ni\mup\to\uromp\in\mathbb{R}^{\rdim}$, where either $\uromp:=\encoder(\ufomp)$ or $\uromp:=\transcoder(\mup, \ufomp)$, depending on the adopted approach. 
\review{As we noted in Section \ref{sec:dlrom}, the reduced map is continuous, as it is given by the composition of $\mup\to\ufomp$ and $\encoder$ (resp. $\transcoder$), hence it can be approximated uniformly by some $\rho$-DNN $\phi:\mathbb{R}^{p}\to\mathbb{R}^{n}$}. In general, we do not impose a particular structure on $\phi$, rather we use a generic fully connected network with dense layers.
To design the architecture in terms of number of layers and neurons, we rely on \review{Theorem 5 and on the underlying} theoretical results available in the literature, e.g. \cite{bolcskei,daubechies,guhring2,park,petersen,siegel}.

\begin{algorithm}[H]
\SetAlgoLined
\normalem
\SetKwInOut{Input}{Input}
\SetKwInOut{Output}{Output}
\vspace{0.25em}
\Input{
Training snapshots $\{\mup_{i}, \ufompi\}_{i=1}^{N}$, reduced dimension $\rdim$, optimizers \texttt{Optimizer}$_{1}$, \texttt{Optimizer}$_{2}$, number(s) of epochs $E_{1}$, $E_{2}$, batch size(s) $S_{1}$, $S_{2}$, encoding type \texttt{useparameters} (boolean), discrepancy measures $\loss_{1}$, $\loss_{2}$.\vspace{0.25em}
}

\Output{
Neural network $\Phi$ approximating the parametric map.\vspace{0.75em}
}

$\encoder_{0}, \decoder_{0}\leftarrow$ Initialize encoder/transcoder and decoder with latent dimension $n$

$e\leftarrow0$ Initialize epochs counter

$B_{1}\leftarrow N/S_{1}$ \quad\textit{// number of batches}

\eIf{\textnormal{\texttt{useparameters}}}{
   \vspace{0.5em}    
   $\vb_{i} \leftarrow [\mup_{i}, \ufompi]$ \quad\textit{// transcoder case}
   }{
   $\vb_{i} \leftarrow \ufompi$ \quad\textit{// encoder case}
  }

 \While{\textnormal{$e < E_{1}$}}{
  \vspace{0.5em}
  shuffle training data $\{\vb_{i}, \ufompi\}_{i=1}^{N}$\;
  \For{\textnormal{$m = 1:B_{1}$}}{
        \vspace{0.5em}
        $\vb^{\text{batch}}\leftarrow[\vb_{(m-1)S_{1}+1},..,\vb_{mS_{1}}]$
        
        \vspace{0.5em}
        
        $\ub^{\text{batch}}\leftarrow[\ufom_{\mup_{(m-1)S_{1}+1}},..,\ufom_{\mup_{mS_{1}}}]$
        \newline
        
        loss $\leftarrow \frac{1}{S_{1}}\sum_{i=1}^{S_{1}}\loss_{1}\left(\ub^{\text{batch}}_{i}, \Psi_{eB_{1}+m-1}(\Psi'_{eB_{1}+m-1}(\vb^{\text{batch}}_{i}))\right)$ 
        
        \vspace{0.5em}
        
        $\encoder_{eB_{1}+m},\decoder_{eB_{1}+m}\leftarrow\texttt{Optimizer}_{1}(\text{loss}, \encoder_{eB_{1}+m-1}, \decoder_{eB_{1}+m-1})$
  }
  $e\leftarrow e+1$
 }
 \vspace{0.75em}
 $\phi_{0}\leftarrow $ Initialize reduced map DNN
 
 $e\leftarrow0$ Reset epochs counter
 
 $B_{2}\leftarrow N/S_{2}$ \quad\textit{// number of batches}
 
 $\urompi\leftarrow \encoder(\vb_{i})$    \quad\textit{// define training data for $\phi$}
 
  \While{\textnormal{$e < E_{2}$}}{
  \vspace{0.5em}
  shuffle training data $\{\mup_{i}, \urompi\}_{i=1}^{N}$\;
  \For{\textnormal{$m = 1:B_{2}$}}{
        \vspace{0.5em}
        $\mup^{\text{batch}}\leftarrow[\mup_{(m-1)S_{2}+1},..,\mup_{mS_{2}}]$
        
        \vspace{0.5em}
        
        $\ub^{n,\text{batch}}\leftarrow[\urom_{\mup_{(m-1)S_{2}+1}},..,\urom_{\mup_{mS_{2}}}]$
        \newline
        
        loss $\leftarrow \frac{1}{S_{2}}\sum_{i=1}^{S_{2}}\loss_{2}\left(\ub^{n,\text{batch}}_{i}, \phi_{eB_{2}+m-1}(\mup^{\text{batch}}_{i})\right)$ 
        
        \vspace{0.5em}
        
        $\phi_{eB_{2}+m}\leftarrow\texttt{Optimizer}_{2}(\text{loss}, \phi_{eB_{2}+m-1})$
  }
  $e\leftarrow e+1$
 }
 \vspace{0.75em}
 $\Phi\leftarrow\decoder_{E_{1}B_{1}}\circ\phi_{E_{2}B_{2}}$
 
 \Return{$\Phi$}

 \caption{DL-ROM training.\vspace{0.25em}}
 
\end{algorithm}

\review{In order to train $\phi$ we minimize the objective function below
$$\lossf(\phi)=\frac{1}{\ntrain}\sum_{i=1}^{\ntrain}\loss(\uri,\;\phi(\mui))$$
where, once again, $\loss$ is some discrepancy measure (this time having inputs in $\mathbb{R}^{n}\times\mathbb{R}^{n})$. Notice that the optimization of the above only involves $\phi$, as the weights and biases of $\encoder$ (resp. $\transcoder)$ are frozen.}

At the end of the whole process, which we summarized in Algorithm 1, we let $\Phi:=\decoder\circ\phi$. Now the DL-ROM is fully operational, and for each new $\mup\in\Theta$ we can approximate online the corresponding solution $\Phi(\mup)\approx\ufomp$ almost effortlessly, with very little computational cost. \review{Also, the model can be efficiently evaluated on multiple parameter values simultaneously. In fact, as DNNs are ultimately based on elementary linear algebra, it possible to stack together multiple parameter vectors $\mathbf{M}=[\mup_{1},\dots,\mup_{l}]$ in a single matrix and directly return the corresponding list of ROM approximations $\Phi(\mathbf{M})\approx[\ub_{1}^{h},\dots,\ub_{l}^{h}]$.}

\noindent\newline\textit{Remark.\;} We mention that, in the case $n=p$, an interesting alternative for $\phi$ could be provided by the so-called ODE-nets \cite{chen}. In fact, if the reduced map happens to be injective, then $\Theta$ and $\{\uromp\}_{\mu\in\Theta}$ define two homeomorphic sets of coordinates. Even though homeomorphisms can be approximated by classical DNNs, we note that fully connected unconstrained networks can easily result in noninvertible models. In this sense, an alternative architecture which ensures the existence and continuity of $\phi^{-1}$ would be appealing. ODE-nets enjoys such property and have been proven to be universal approximators for homeomorphisms \cite{zhang}. However, the development and implementation of ODE-nets is still in its infancy so we did not investigate this further.

\section{Numerical experiments}
\label{sec:experiments}
We now present some numerical results obtained with our DL-ROM approach. So far, neural networks have shown remarkable performances in the approximation of the parametric map at least in those contexts where classical POD-based methods succeed, e.g. \cite{bhattacharya,geist}. There is now an increasing interest in understanding how and if NNs can be of help in more challenging situations. In the case of transport problems, some theoretical and numerical results are now appearing in the literature, see respectively \cite{laakman} and \cite{fresca}.

Here, we focus on parameter dependent second order elliptic PDEs. The first test case concerns an advection-diffusion problem with a singular source term. The PDE depends on 7 scalar parameters which affect the equation both in a linear and nonlinear fashion. We consider two variants of the same problem, one of which is transport-dominated. As second test case, we consider a stochastic Poisson equation. The main difference with respect to the previous case is that the equation is parametrized by a stochastic process, and the PDE formally depends on an infinite-dimensional parameter. In order to apply the DL-ROM approach, we consider a suitable truncation of the Karhunen–Loève expansion of the stochastic process.

\review{In all our experiments we consider $V=L^{2}(\Omega)$ as state space, and we quantify the ROM performance via the Mean Relative Error (MRE)
\begin{equation}
    \label{eq:mre}
    \mathbb{E}_{\mup\sim\mathcal{P}}\left[\frac{||\ufomp-\Phi(\mup)||_{L^{2}(\Omega)}}{||\ufomp||_{L^{2}(\Omega)}}\right],
\end{equation}
where $\Phi=\Psi\circ\phi$ is the DL-ROM network, and $\mathcal{P}$ is some probability measure defined over the parameter space $\Theta$. We estimate \eqref{eq:mre} with a Monte Carlo average computed over 1000 unseen snapshots (test set). \final{To evaluate whether there is a gap in performance between training and testing, we also compute the MREs over the training set. For an easier comparison, in all our experiments, we fix the latter to have size $\ntrain=9000$.}}
\newline\newline
All our experiments were implemented in Python 3 and ran over GPUs. Specifically, we used the FEniCS library\footnote{https://fenicsproject.org/} to run the FOM and obtain the high-fidelity snapshots, while the construction and the training of the DL-ROM was handled in Pytorch\footnote{https://pytorch.org/}.

\subsection{Stationary advection-diffusion with singular source}
\label{sec:exp1}

\subsubsection{Problem definition}
On the spatial domain $\Omega = (0,1)^{2}$, we define the subdomains $\{\Omega_{i}\}_{i=0}^{4}$ as in Figure \ref{fig:domain-decomp}. We consider the following parameter dependent PDE in weak form
\[
\begin{multlined}
\frac{1}{10}\int_{\Omega_{0}}\nabla u\cdot\nabla w
+\sum_{i=1}^{4}\int_{\Omega_{i}}\mu_{i}\nabla u\cdot\nabla w
+ C\int_{\Omega}\left(\cos{\mu_{5}}\frac{\partial u}{\partial x_{1}}w
+ \sin{\mu_{5}}\frac{\partial u}{\partial x_{2}}w\right) = \\\\ = w(\mu_{6},\mu_{7})\quad\forall w\in \mathcal{C}^{\infty}_{0}(\Omega)\quad\quad
\end{multlined}\]
\review{with Dirichlet boundary condition $u_{|\Omega}=1$.} The above corresponds to a stationary advection-diffusion equation where: the conductivity field $\sigma_{\mup}:=0.1 + \sum_{i=1}^{4}\mu_{i}\textbf{1}_{\Omega_{i}}$ is piecewise constant with values that change parametrically within the circular subdomains; the transport field $\advep:=(C\cos\mu_{5},C\sin\mu_{5})$ has a parametrized direction while it is uniform in space and has a fixed intensity $C>0$; finally, the source term $\forcep$ is a Dirac delta located at the parameter dependent coordinates $(\mu_{6}, \mu_{7})$. \review{Globally, the PDE depends on 7 parameters that we consider to be varying in the parameter space $\Theta = [0,1]^{4}\times[0, 2\pi]\times[0.1, 0.9]^{2}$, which we endow with a uniform probability distribution $\mathcal{P}$}. We note that the PDE does not admit solutions in $H^{1}(\Omega)$ because of the singularity introduced by the Dirac delta. Nevertheless, the variational problem is well-posed in the Banach space $W^{1,4/3}(\Omega)\hookrightarrow L^{2}(\Omega)$, see e.g. \cite{casas}. We are hence allowed to consider the solution manifold $\solmanifold:=\{\up\}_{\mup\in\Theta}$ as a subset of the Hilbert space $L^{2}(\Omega)$.

We analyze two different settings. In the first case we fix the transport field intensity to be $C=0.5$, so that the diffusion and the advection act over the same scale. Then, we consider a transport-dominated case where $C=40$.
\begin{figure}
  \begin{minipage}[c]{0.4\textwidth}
    \includegraphics[width=\textwidth]{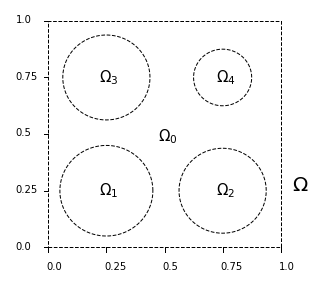}
  \end{minipage}\hfill
  \begin{minipage}[b]{0.6\textwidth}
    \caption{
       Decomposition of the unit square %
       $\Omega=(0,1)^{2}$ according to the conventions adopted in the first numerical experiment, Section \ref{sec:exp1}.
       \vspace{-5.5em}
    } \label{fig:domain-decomp}
  \end{minipage}
\end{figure}

\subsubsection{Discretization and Full Order Model}
As FOM, we employ Lagrange piecewise linear finite elements over a triangular mesh. Prior to the discretization, we provide a Gaussian approximation of the Dirac delta as
$$\forcep^{\epsilon}(x_{1},x_{2}):=\frac{1}{2\pi\epsilon^{2}}\exp\left({\frac{-(x_{1}-\mu_{6})^{2}-(x_{2}-\mu_{7})^{2}}{2\epsilon^{2}}}\right).$$
We shall write $\up^{\epsilon}$ for the solutions of this smoothed problem and $\solmanifold^{\epsilon}$ for the corresponding solution manifold. We see that the following claim holds (for the details see the Appendix, Section \ref{appendix:exp})
\\\\\noindent\textbf{Claim 1. }$\sup_{\mup\in\Theta}||\up-\up^{\epsilon}||_{L^{2}(\Omega)}\to0$ as $\epsilon\to0$.\\\\
\noindent In particular, $\solmanifold^{\epsilon}$ approximates $\solmanifold$ uniformly.
From here on, we shall fix $\epsilon = 1/420$ and formally replace $\solmanifold$ with $\solmanifold^{\epsilon}$. Next, we discretize the variational problem through P1-Finite Elements over a triangular mesh. Using the classical estimates from the FEM theory, e.g. \cite{quarteroni-fem}, it is not hard to see that Assumption \ref{assump:high-fidelity} is satisfied within the state space $L^{2}(\Omega)$. Here, we fix the mesh size to be $h = 1/210$, which results in a high-fidelity space $V_{h}\cong\mathbb{R}^{N_{h}}$ of dimension $N_{h}=44521$. We exploit the FOM to generate respectively $N_{train}=9000$ and $N_{test}=1000$ random snapshots. 

\subsubsection{DL-ROM design and training}
In the construction of the DL-ROM, we do not make a distinction between the case of mild and strong advection, respectively $C = 0.5$ and $C=40$. In this way, we can see more clearly whether the intensity of the transport field affects the ROM performance. For the dimensionality reduction, we explore both the two alternatives presented in Section \ref{sec:ae}. For the autoencoder, we choose to consider the original solution manifold as a reference for the latent dimension, thus we let $\rdim:=\mindim(\solmanifold)$. Thanks to Theorem \ref{theorem:parametric-reduction}, the claim below holds true. 
\\\\\noindent\textbf{Claim 2. } $\mindim(\solmanifold) = p = 7$.\\\\
\noindent The proof is straightforward and we leave the details to the Appendix, Section \ref{appendix:exp}. Note that in this way, regardless of the encoding strategy, we are fixing the reduced dimension to be $n=7$. Since $N_{h}=44521$, this corresponds to a compression of almost 99.98\%. 

{\small
\begin{table}
\caption{\label{tab:architectures1} 
Architectures for the Advection-Diffusion problem (Section \ref{sec:exp1}). Tables (a) and (b) refer to the encoder and transcoder, respectively. Table (c) reports the decoder architecture. Here Transp. Conv. = Transposed Convolutional, denotes those sparse layers whose linear part is described in terms of 2D transposed convolutions. Therein, the input and output dimensions are written in the form $channels\times height\times width$, as each vector is reshaped in a 3D tensor. The hyperparameters kernel (which stands for \textit{kernel size}) and \textit{stride} determine the characteristics of the transposed convolution. For a more detailed explanation we refer to the Pytorch library documentation. Note that, up to reshaping, the output has dimension $1\cdot211\cdot211=N_{h}$. The network architecture is given in terms of a hyperparameter $m\in\mathbb{N}$, through which we tune the decoder complexity. Table (d) refers to the change of coordinates DNN, $\phi$. All networks and layers use the 0.1-leaky ReLU activation.}

\begin{flushleft}
a) $\encoder\hspace{15em}$b) $\transcoder$\vspace{0.25em}\newline
 \begin{tabular}{||l l l l||} 
 \hline
 Layer & Input & Output & dof\\ [0.5ex] 
 \hline\hline
 Dense & 44521 & 7 & 311654\\ 
 \hline
\end{tabular}
\quad
 \begin{tabular}{||l l l l||} 
 \hline
 Layer & Input & Output & dof\\ [0.5ex] 
 \hline\hline
 Dense & 44528 & 7 & 311703\\ 
 \hline
\end{tabular}

\vspace{0.5em}
 c) $\decoder$\vspace{0.25em}
 \newline
 \begin{tabular}{||l l l l l l||} 
 \hline
 Layer & Input & Output & Kernel & Stride & dof\\ [0.5ex] 
 \hline\hline
 Dense & 7 & $288m$ & - & - & $2304m$\\ 
 \hline
 Transp. Conv. & $8m\times6\times6$ & $4m\times20\times20$ & 10 & 2 & $3200m^{2}+4m$\\ 
 \hline
 Transp. Conv. & $4m\times20\times20$ & $2m\times48\times48$ & 10 & 2 & $800m^{2}+2m$\\ 
 \hline
 Transp. Conv. & $2m\times48\times48$ & $m\times102\times102$ & 8 & 2 & $128m^{2}+m$\\ 
 \hline
 Transp. Conv. & $m\times102\times102$ & $1\times211\times211$ & 9 & 2 & $81m+1$\\ 
 \hline
\end{tabular}

\vspace{0.5em}
d) $\phi$
\vspace{0.25em}
\newline
\begin{tabular}{||l l l l||} 
 \hline
 Layer & Input & Output & dof\\ [0.5ex] 
 \hline\hline
 Dense & 7 & 1024 & 7168\\ 
 \hline
 Dense & 1024 & 512 & 524288\\ 
 \hline
 Dense & 512 & 256 & 131072\\ 
 \hline
 Dense & 256 & 7 & 1792\\ 
 \hline
\end{tabular}
\end{flushleft}
\end{table}
}

The networks architectures are detailed in Tables \ref{tab:architectures1}.a, \ref{tab:architectures1}.b (encoding step) and Table \ref{tab:architectures1}.c (decoding step). The encoder and the transcoder are particularly light, as they actually consist of a single dense layer. In contrast, $\decoder$ is far more complex, with a depth of $l=4$. \review{The proposed architecture is closely related to the ones adopted in \cite{fresca,lee}, upto to some specifics dictated by the problem itself (namely the fact the high-fidelity mesh consists of a $211\times211$ square)}. The decoder makes use of transposed convolutional layers, choice that is mainly motivated by two reasons: (i) convolutional layers correspond to sparse operators, and are more easy to deal with in the case of high-dimensional data, (ii) 2D convolutions best describe spatially localized behaviors so they are a natural choice when the data itself is defined on a spatial domain. We shall also remark that the decoder architecture is given in terms of a hyperparameter $m\in\mathbb{N}$, $m>0$, which controls the number of channels in the convolutional layers. This was done in order to investigate how the network complexity impacts the reconstruction error, and allows for a direct comparison with linear methods such as the POD. We analyze the performance of the networks for different values of $m$ separately, namely $m=4,8,16,32$.
\newline\indent
Prior to training, the networks are initialized differently depending on the encoding type. In the autoencoder case, we initialize both $\encoder$ and $\decoder$ accordingly to the (Gaussian) He initialization \cite{he}. Conversely, we initialize the transcoder in such a way that $\transcoder(\mup,\up)=\mup.$ This is equivalent to using the parameters as first guess for the intrinsic coordinates: then, during the training, and depending on the decoder needs, $\transcoder$ will have the possibility of finding other representations. \review{As discrepancy measure for the loss function, we use squared errors, $\mathcal{L}(\y,\hat{\y})=||\y-\hat{\y}||^{2}$}. We train the autoencoder (resp. transcoder-decoder) using the AdamW optimizer \cite{loshchilov}, with learning-rate $10^{-4}$, weight-decay $10^{-2}$, moments coefficients $\beta_{1}=0.99$, $\beta_{2}=0.999$ and adjustment $\varepsilon=10^{-8}$. We perform the gradient descent using batches of 50 and for a total of 1000 epochs. At the end of this first training session, we pick the best performing architecture and continue the construction of the DL-ROM from there.

Table \ref{tab:architectures1}.d reports the architecture for our third network, $\phi$. We initialize $\phi$ using the He initialization and proceed with its training accordingly to Section \ref{sec:cc}. \review{We consider again a loss function based on square errors}, and we perform the gradient descent using the same optimizer as before, only changing the learning rate to $10^{-3}$.

\subsubsection{Numerical results}
\begin{figure}
\includegraphics[width=\textwidth]{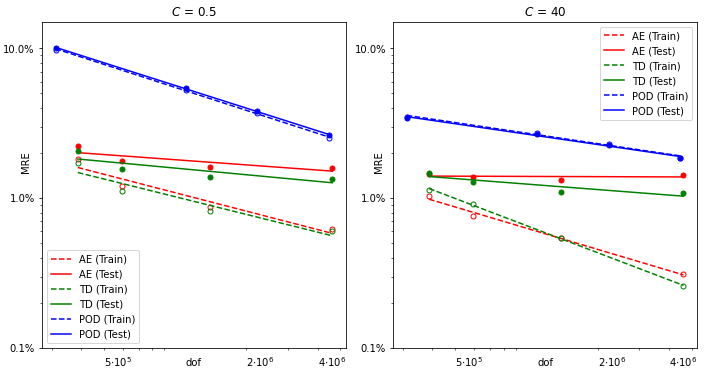}
\caption{
    Error decay in terms of network complexity for the Advection-Diffusion problem (Section \ref{sec:exp1}). Plots are reported in loglog scale. Case $C=0.5$ on the left, $C=40$ on the right. Lines are drawn by considering the least-square fit $\log\text{MRE} \approx \beta_{0}+\beta_{1}\log\text{dof}$. Dashed-lines are used for training errors and straight lines for test errors. In the case of autoencoders (AE, in red) and transcoder-decoders (TD, in green) the reported MREs correspond to the architectures in Table \ref{tab:architectures1} with $m=4,8,16,32$. POD-projection errors are reported in blue. POD degrees of freedom (dof) are computed as $nN_{h}$ and correspond to the number of entries in the projection matrix $\mathbf{V}$. Note that for POD the reduced dimension increases as the complexity grows. Conversely, the DL-ROM approach considers heavier and heavier networks, but the reduced dimension is always fixed to $n=p=7$.
}
\label{fig:dimreduction}
\end{figure}
\review{Figure \ref{fig:dimreduction} reports the results limited to the dimensionality reduction, that is the first step in the DL-ROM pipeline. There, we compare the performance of autoencoders, transcoder-decoders and POD in terms of model complexity. In general, regardless of whether $C=0.5$ or $C=40$, both nonlinear methods show interesting results, with training errors close or below 1\%. Unsurprisingly, as the networks grow in complexity, the gap between training and test errors becomes larger, highlighting the need for more samples and a tendency towards overfitting (see the autoencoder in case $C=40$). Still, transcoders seem to mitigate this phenomenon, possibly because they provide more information in the latent space.

For POD, the degrees of freedom are defined as the number of entries in the projection matrix $\mathbf{V}$, while the errors are computed as $||\ub-\mathbf{V}\mathbf{V}^{T}\ub||/||\ub||$ (relative projection error). In particular, the MREs reported in Figure \ref{fig:dimreduction} provide a lower-bound for all POD-based ROMs. Interestingly, all the curves show a similar trend. This goes to support our conjecture that the decoder complexity may be linked with the Kolmogorov $n$-width (Section \ref{sec:dlrom}). More precisely, linear methods can improve the accuracy by adding $\Delta n$ modes, i.e. $\Delta n N_{h}$ degrees of freedom, but they also have to increase the ROM dimension. Conversely, in the DL-ROM approach, we can obtain a similar boost by investing the same degrees of freedom in the decoder, without having to modify the latent dimension. This is in agreement with Theorem 5, where we proved that $\mathcal{O}(m^{1+n/(s-1)}\log(m))$ active weights are sufficient for the decoder to match the accuracy of any projection method with $m$ modes. Of note, if we assume the solution manifold to be infinitely smooth and we let $s\to+\infty$, then we may conjecture the decoder complexity to behave as $\mathcal{O}(m\log(m))$. As matter of fact, this is what we observe in the picture, at least for the training errors. In fact, if the red lines were to be perfectly parallel to the blue ones, that would reflect a scenario in which the decoder complexity grows as $\mathcal{O}(m)$.

On the same time, Figure \ref{fig:dimreduction} goes to show that the upper bounds in Theorem 5 are suboptimal. In fact, both in the case of mild and strong advection, the nonlinear reduction is able to outperform POD with fewer degrees of freedom, i.e. without the extra contribute $m\fomdim$ in the decoder. We believe that this achievement is made possible by the use of convolutional layers, but this is yet to be made formal.

\newcommand{\npod}{n_{\text{POD}}}

{\small
\begin{table}
 \begin{tabular}{||l l l l||} 
 \hline
$C$ & Data & MRE & Equivalent POD modes\\ [0.5ex] 
 \hline\hline
 0.5 & Train & 1.05\% & 428\\ 
 \hline
 0.5 & Test  & 2.01\% & 164\\ 
 \hline
 40  & Train & 0.31\% & more than 1000\\ 
 \hline
 40  & Test  & 1.23\% & 316\\ 
 \hline
\end{tabular}
\caption{\label{tab:errors1} \review{DL-ROM performance for the Advection-Diffusion problem (Section \ref{sec:exp1}). The final model $\Phi:=\Psi(\phi)$ was constructed by choosing the decoder with highest performance (cf. Figure \ref{fig:dimreduction}). Equivalent POD modes = minimum number of POD-modes needed by any POD-based ROM to outperform the DL-ROM.}}
\end{table}
}
\begin{figure}
\begin{flushleft}
\footnotesize{
a) Case $C=0.5$,\hspace{14.5em}b) Case $C=40$
}
\end{flushleft}
\includegraphics[width=\textwidth]{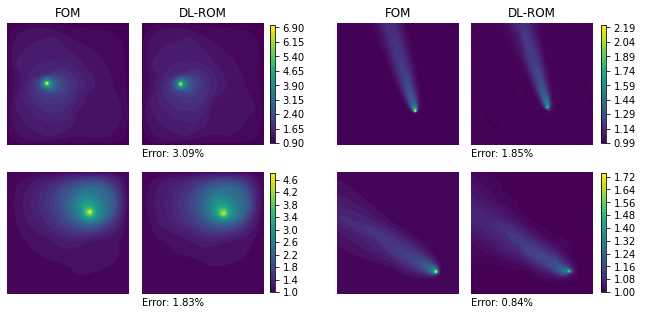}
\caption{
    DL-ROM results for the Advection-Diffusion problem (Section \ref{sec:exp1}). Panels (a) and (b) respectively refer to the case of mild and strong advection. In both cases, the picture shows two examples extracted from the test set (one for each row) and compares the high-fidelity solution (first column) with the DL-ROM approximation (second column). 
}
\label{fig:examples1}
\end{figure}
Let us now move to the actual approximation of the parametric map. To this end, we trained our third network $\phi$ on the basis of the best performing transcoder-decoder ($m=32$). Numerical results for the complete DL-ROM $\Phi:=\Psi\circ\phi$ are in Table \ref{tab:errors1} and Figure \ref{fig:examples1}. In general, the results are satisfactory, with test errors near 2\%. Both in the case of mild and strong advection, we note that POD-based ROMs require more than 300 modes to achieve the same accuracy. This makes intrusive ROMs, such as POD-Galerkin, too expensive to be used online. Conversely, the DL-ROM approach provides an appealing alternative. Indeed, while the whole offline stage took around 4 hours, the model is extremely fast when used online: solving the PDE for 1000 different values of the parameters (simultaneously) requires less than 2 milliseconds on GPU.}

\subsection{Stochastic Poisson equation}
\label{sec:exp2}
\subsubsection{Problem definition} On the spatial domain $\Omega = (0,1)^{2}$, we consider a Gaussian process $W$ with constant mean $w = -\log(10)$ and covariance kernel $\text{Cov}(\x,\y)=10\text{exp}(-4|\x-\y|^{2})$.
The latter is used to model the stochastic Poisson equation below,

$$
\begin{cases}
-\text{div}\left(\text{e}^{W(\omega)}\nabla u\right) = |\x|^{2} & \text{in }\Omega,
\\
u = 0  & \text{on }\partial\Omega.
\end{cases}
$$
\newline
Here, for each event $\omega$, the map $W(\omega):\Omega\to\mathbb{R}$ denotes the corresponding path of the stochastic process $W$. The above problem can be seen as a parameter dependent PDE that depends on (countably) infinite many parameters. To see this, we recall that there exist positive real numbers $\left\{\sqrt{\lambda_{i}}\right\}_{i=1}^{+\infty}$, orthonormal functions $\{\zeta_{i}\}_{i=1}^{+\infty}\subset L^{2}(\Omega)$ and independent standard gaussians $\{X_{i}\}_{i=0}^{+\infty}$ such that $$W = w+\sum_{i=1}^{+\infty}\sqrt{\lambda_{i}}X_{i}\zeta_{i}$$ almost surely. The latter is the so-called Karhunen-Loève expansion of $W$. We assume the $\lambda_{i}$ coefficients to be nonincreasing in $i$. In order to cast the problem into our framework, we approximate $W$ by truncating the aforementioned expansion at some index $k$. More precisely, we define $\Theta:=\mathbb{R}^{k}$ and $W_{\mup}^{k}$ as

$$W_{\mup}^{k}(x):= w+\sum_{i=1}^{k}\sqrt{\lambda_{i}}\mu_{i}\zeta_{i}(x).$$
\newline
Thanks to the usual continuity results and the convergence ensured by the  Karhunen-Loève expansion, the impact of this substitution on the PDE can be made arbitrarily small with $k$. \review{We note that, by construction, the probability distribution to be considered over the parameter space is the Gaussian distribution $\mathcal{P}$ of density
$$G(\mup):=\left(2\pi\right)^{-k/2}\text{e}^{-\frac{1}{2}|\mup|^{2}}.$$}

\subsubsection{Discretization and Full Order Model}

On $\Omega$ we define a triangular mesh of size $h=10^{-2}$, over which we construct the high-fidelity space of piecewise linear Finite Elements $V_{h}$. The corresponding FOM dimension is $N_{h}=10121$. To approximate the Karhunen-Loève expansion of $W$, we project and solve over $V_{h}$ the following eigenvalue problem.

$$\int_{\Omega}\text{Cov}(\x,\y)\zeta_{i}(\y)d\y = \lambda_{i}\zeta_{i}(\x).$$
\newline
In particular, we compute the first $k$ eigenvalues $\lambda_{i}$ and corresponding eigenfunctions $\zeta_{i}\in V_{h}$ for which

$$0.9\;\le\;\frac{\sum_{i=1}^{k}\lambda_{i}}{\sum_{i=1}^{+\infty}\lambda_{i}} = \frac{1}{10}\sum_{i=1}^{k}\lambda_{i},$$
\newline
where the last equality is easily deduced by the covariance kernel, as $\sum_{i=1}^{+\infty}\lambda_{i} = \int_{\Omega}\text{Cov}(\x,\x)d\x=\int_{[0,1]^{2}}10d\x = 10$. \begin{figure}
    \centering
    \includegraphics[width=0.7\textwidth]{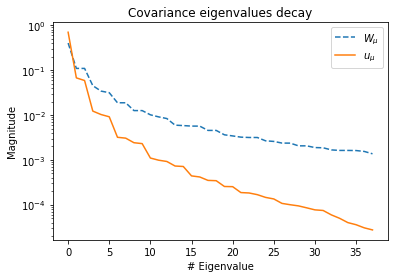}
    \caption{\review{Normalized eigenvalues decay for the covariance matrices of the parameters, i.e. the random field $W_{\mup}$, and the solutions $u_{\mup}$. The eigenvalues are normalized with respect to the total variance of the two fields, so that their cumulative sum converges to 1. The y-axis is reported in logarithmic scale. Note that the solutions are sampled by considering the truncated field $W_{\mup}^{k}$ as parameter.}}
    \label{fig:eigdecay}
\end{figure} \review{Figure \ref{fig:eigdecay} shows how the normalized eigenvalues $\lambda_{k}/10$ decay with $k$.} This procedure results in the choice of the truncation index $k=38$.
From a statistical point of view, we say that $W_{\mup}^{k}$ explains at least 90\% of the variability in $W$. We run the FOM to generate respectively $\ntrain=9000$ and $\ntest=1000$ snapshots, where the parameter values are sampled from $\Theta$ independently and accordingly to a $k$-variate standard Gaussian distribution.

\subsubsection{DL-ROM design}
We note that, in this case, the parameter space $\Theta$ is not compact, as it is unbounded. Nevertheless, since $\Theta$ has finite measure with respect to $\mathcal{P}$, it is straightforward to adapt the reasoning in Section \ref{sec:dlrom} to this context. For instance, the error defined in \eqref{eq:mre} can be made arbitrarily small provided that $\Phi$ is sufficiently accurate within some compact subdomain $\Theta_{M}:=\{\mup\in\mathbb{R}^{p}\text{ s.t. }|\mup|<M\}$. A more in depth discussion on the regularity of the parametric map in the case of stochastic coefficients can be found in \cite{babuska}.

\review{For the dimensionality reduction, we employ a transcoder-decoder. This is to ensure a maximal compression, as the number of parameters is already mildly large. The network topology is reported in Tables \ref{tab:architectures2}.a, \ref{tab:architectures2}.b. Coherently with the chosen approach, we fix the reduced dimension to be $\rdim:=k=p=38$. In general, the architecture is very similar to those considered in Section \ref{sec:exp1}, the only difference being in the specifics of the convolutional layers.
\newline\indent
As before, we adopt the He initialization for the decoder while we force the initial state of the transcoder to behave as $\transcoder(\mup,\up)=\mup$. We train $\decoder\circ\transcoder$ using stochastic gradient descent with minibatches of size 10 and for a total of 1200 epochs. In this case, we observe that snapshots come in rather different scales when compared one another. For this reason, we choose to define the loss function in terms of relative errors, $\loss(\y,\hat{\y}):=||\y-\hat{\y}||/||\y||.$ To optimize the latter we employ the Adamax optimizer \cite{kingma}, with default parameters and learning rate of $10^{-3}$. Here, the choice of Adamax over AdamW is motivated by the fact that the former is known to be more stable.
\newline\indent
Table \ref{tab:architectures2}.c reports the architecture for reduced map network, $\phi$. We train $\phi$ using the Adamax optimizer with batch size 50 and learning rate 5e-3, for a total of 5000 epochs. Here also we use relative errors as discrepancy measures. The whole offline stage of the DL-ROM took around 4 hours.}
{\small
\begin{table}
\caption{\label{tab:architectures2} Architectures for the Stochastic Poisson equation (Section \ref{sec:exp2}). Tables (a) and (b) together describe the transcoder-decoder, while (c) concerns the change of coordinates map $\phi$. All layers are considered with the 0.1-leaky ReLU activation.}
\begin{flushleft}
a) $\transcoder$\vspace{0.25em}
\newline
 \begin{tabular}{||l l l l||} 
 \hline
 Layer & Input & Output & dof\\ [0.5ex] 
 \hline\hline
 Dense & 10239 & 38 & 389120\\ 
 \hline
\end{tabular}
\end{flushleft}
\vspace{1em}
\begin{flushleft}
b) $\decoder$\vspace{0.25em}
\newline
  \begin{tabular}{||l l l l l l||} 
 \hline
 Layer & Input & Output & Kernel & Stride & dof\\ [0.5ex] 
 \hline\hline
 Dense & 38 & 18432 & - & - & 718848\\ 
 \hline
 Transp. Conv. & $512\times6\times6$ & $256\times12\times12$ & 2 & 2 & 524544\\ 
 \hline
 Transp. Conv. & $256\times12\times12$ & $128\times24\times24$ & 2 & 2 & 131200\\ 
 \hline
 Transp. Conv. & $128\times24\times24$ & $64\times48\times48$ & 2 & 2 & 32832\\ 
 \hline
 Transp. Conv. & $64\times48\times48$ & $1\times101\times101$ & 7 & 2 & 3137\\ 
 \hline
\end{tabular}
\end{flushleft}
\vspace{1em}
\begin{flushleft}
c) $\phi$ \vspace{0.25em}
\newline
\begin{tabular}{||l l l l||} 
 \hline
 Layer & Input & Output & dof\\ [0.5ex] 
 \hline\hline
 Dense & 7 & 1024 & 7168\\ 
 \hline
 Dense & 1024 & 512 & 524288\\ 
 \hline
 Dense & 512 & 256 & 131072\\ 
 \hline
 Dense & 256 & 7 & 1792\\ 
 \hline
\end{tabular}
\end{flushleft}
\end{table}
}

\subsubsection{Numerical results}
The dimensionality reduction is satisfactory, with mean relative errors of 1.10\% and 2.57\% respectively on the training and test sets. \review{Conversely, the approximation of the reduced map was more challenging, see Table \ref{tab:errors2} and Figure \ref{fig:examples2}. While the final model is able to approximate the parameter-to-state map with an error of 4.69\% over the training set, the inaccuracy increases to 12.50\% on the test set. This is a situation in which the solution manifold is relatively simple (even linear subspaces provide good approximations, cf. Figure \ref{fig:eigdecay}), but the parameter dependency is complicated. Therefore, while 9000 snapshots are sufficient for the training the transcoder-decoder, they are not enough for $\phi$ to generalize well. Another reason is that the parameter space is very large, and $\phi$ has to face the curse of dimensionality. Possible ways to overcome this drawback without having to generate more samples would be to exploit low-discrepancy sequences, as in \cite{mishra2}, or use physics-informed approaches at the reduced level, as in \cite{hesthaven3}.}

{\small
\begin{table}
 \begin{tabular}{||l l l||} 
 \hline
Data & MRE & Equivalent POD modes\\ [0.5ex] 
\hline\hline
Train & 4.69\% & 39\\ 
\hline
Test  & 12.50\% & 17\\ 
\hline
\end{tabular}
\caption{\label{tab:errors2} \review{DL-ROM performance for the Stochastic Poisson problem (Section \ref{sec:exp2}). Equivalent POD modes = minimum number of POD-modes needed by any POD-based ROM to outperform the DL-ROM.}}
\end{table}
}
\begin{figure}
\includegraphics[width=0.9\textwidth]{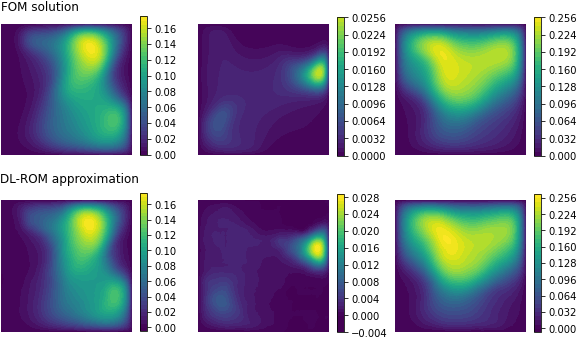}
\newline
\vspace{-1.5em}
\begin{flushleft}
\tiny{
\hspace{1.6em}Relative error: 5.90\%
\hspace{5.4em}Relative error: 23.30\%
\hspace{4.5em}Relative error: 4.19\%
}
\end{flushleft}
\caption{
    DL-ROM results for the Stochastic Poisson equation (Section \ref{sec:exp2}). The picture shows three examples coming from the test set. The first row reports the high-fidelity solutions, while the second row displays the corresponding DL-ROM approximations. Relative errors are also reported.
}
\label{fig:examples2}
\end{figure}

\newpage
\subsection{Some final remarks}

\subsubsection{On the choice of the latent dimension}
\label{sec:remark}
In all our experiments, the reduced dimension was equal to the number of parameters, that is $n=p$. For the Advection-Diffusion problem, this was motivated by the fact that $\mindim(\solmanifold)=p$. Conversely, for the stochastic Poisson equation, we fixed $n=p$ due to our design choice of using a transcoder-decoder. However, this is not always the case. In fact, as we argued in Section \ref{sec:nonlinear}, it is possible that $\mindim(\solmanifold)>p$, in which case the autoencoder latent dimension should exceed the number of parameters. As an example, consider the following boundary value problem,
\begin{equation}
\label{eq:remark}
\begin{cases}
-\Delta u 
+ 10\left(\cos\mu, \sin\mu\right)\cdot\nabla u
= 
10\text{e}^{-100|\x-\x_{0}|}
& \text{in }\Omega
\\
u = 0  & \text{on }\partial\Omega,
\end{cases}
\end{equation}
where $\Omega:=(0,1)^{2}$, $\x_{0}:=(0.5, 0.5)$ and $\mu\in\Theta:=[0,2\pi]$. In this case, $p=1$ but $\mindim(\solmanifold)=2$. In fact, it is not hard to see that the solution manifold $\solmanifold$ is homeomorphic to the unit circle. Therefore, the DL-ROM approach requires an autoencoder with latent dimension $n=2$. In Figure \ref{fig:remark} we have summarized the results obtained with this design choice for this particular problem. We do not report the network architectures, as that is not the focus of our discussion here.

We note that the low-dimensional representation of $\solmanifold$ is given by a curve in $\mathbb{R}^{2}$, coherently with the fact that $\phi:\mathbb{R}\to\mathbb{R}^{2}$. It is interesting to see that the DL-ROM representation of $\solmanifold$ actually resembles a circle. We also note that the curve $\phi(\Theta)\subset\mathbb{R}^{2}$ is not smooth. This is not caused by the PDE itself, which is very regular, by rather by the use of the ReLU activation for $\phi$.

For the sake of completeness, we mention that: for this example, we considered the same FOM as in Section \ref{sec:exp2}; we trained and tested the DL-ROM respectively over 900 and 100 snapshots; the model reported average relative errors below 2\% both on the training and test set; the offline stage took around 10 minutes on GPU.

\begin{figure}
  \begin{minipage}[t]{0.50\textwidth}
    \includegraphics[width=\textwidth]{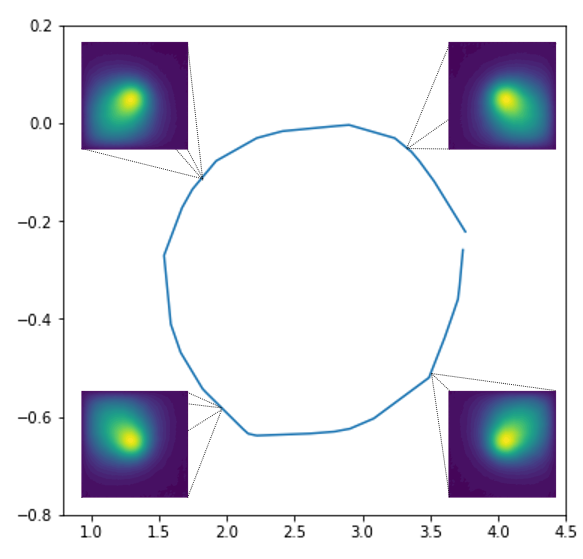}
  \end{minipage}\hfill
  \begin{minipage}[b]{0.50\textwidth}
    \caption{
Low-dimensional representation of the solution manifold for problem \eqref{eq:remark}, Section \ref{sec:remark}. Each point on the curve $\phi(\Theta)\subset\mathbb{R}^{2}$ corresponds to a given solution of the parametrized PDE. The picture also shows the reconstructions $\decoder(\phi(\mu))$ obtained for different values of $\mu\in\Theta$.
    \vspace{6.5em}
    } \label{fig:remark}
  \end{minipage}
\end{figure}

\label{sec:directapproximation}
\subsubsection{On alternative Deep Learning approaches}
\review{The Reduced Order Modelling literature is becoming more and more flourishing, with a large variety of techniques being developed. It is thus important to understand how the DL-ROM approach relates to other Deep Learning based ROMs. In particular, we would like to comment on those alternative strategies that are significantly different in spirit.

One approach, is to directly approximate the correspondence $(\mup,\x)\to u_{\mup}(\x)$ using a single DNN, namely $\Phi:\mathbb{R}^{p+d}\to\mathbb{R}$. This has the advantage of yielding a mesh-free ROM that can be trained on pointwise observations. Also, since $\Phi$ is scalar-valued, light architectures are expected to be sufficiently expressive. However, this approach has a few drawbacks. In fact, if $u_{\mup}$ is not highly regular with respect to $\x$, e.g. not Lipschitz, then one may require a very deep (or large) architecture to obtain reliable approximations, and thus many samples too. For example, we have tested this approach on the Advection-Diffusion problem using a scalar-valued architecture with 7 hidden layers of constant width (50 neurons). We have used the same snapshots available for the DL-ROM, where each high-fidelity solution now contributed with a total of $N_{h}$ observations. Despite all our efforts, we were unable to obtain a sufficiently accurate network, as we always obtained relative errors above 20\% (both for the training and test sets). We believe that the main drawback is given by the singular source, which generates high gradients in the PDE solution. We also mention that, despite the light architecture, training $\Phi$ was quite expensive. In fact, each computation of the squared error $||u_{\mup}^{h}-\Phi(\mup,\cdot)||_{L^{2}(\Omega)}^{2}$ required $N_{h}$ evaluations of the DNN. Even though these can be computed in parallel by stacking all quadrature nodes in a single matrix, the cost of the backpropagation step increases substantially, as keeping track of all the gradients becomes challenging.

An alternative strategy is given by DeepONets \cite{karniadakis}, which are now becoming very popular. DeepONets are primarily employed for learning operators in infinite dimensions, but they have a natural adaptation to the case of finite-dimensional parameters. In fact, the first step in the DeepONet pipeline is to encode the input through $p$ sensors, which allows us to formally recast the problem into one with $p$ scalar parameters. With this set up, DeepONets are still a mesh-free approach, but they consider an approximation of the form $u_{\mup}(\x)=\Psi(\x)\cdot\phi(\mup)$, where $\cdot$ is the dot product, while $\Psi:\mathbb{R}^{d}\to\mathbb{R}^{n}$ and $\phi:\mathbb{R}^{p}\to\mathbb{R}^{n}$ are two neural networks, respectively called the \textit{trunk net} and the \textit{branch net}. The main advantages of DeepONets are the following. First of all, as they are intrinsically mesh-free, it is possible to train them on sparse pointwise data. Secondly, as they decouple the dependency between $\mup$ and $\x$, it is possible to bound their complexity and to estimate their generalization capabilities in ways that are specific to this approach, see e.g. \cite{mishra3}. Finally, due to their original construction, they can be a natural choice when the input parameters are actually sensor observations of some functional input. However, DeepONets have their limitations as well. In fact, despite sharing some terminology with the DL-ROM approach, such as \textit{encoder} and \textit{decoder}, they ultimately rely on a linear strategy for representing solutions. To this end, let $\{\x_{i}\}_{i=1}^{\fomdim}$ be the nodes in the high-fidelity mesh. Then, the DeepONet approximation over these vertices is $\mathbf{V}\phi(\mup)$, where $\mathbf{V}:=[\Psi(\x_{1}),\dots,\Psi(\x_{\fomdim})]^{T}\in\mathbb{R}^{\fomdim\times n}$. As a consequence, the choice of $n$ is subject to the behavior of the linear Kolmogorov $n$-width. For instance, to match the DL-ROM accuracy in the Advection-Diffusion problem with $C=40$, a DeepONet architecture would require $n\ge300$, which may hinder its actual implementation. Also, due to the poor regularity with respect to the $\x$ variable, training $\Psi$ may be a challenging task. 

Conversely, the DL-ROM approach treats solutions as single objects, $\ufomp\in V_{h}$. While this clearly results in a loss of information, the space dependency of solutions can be partially recovered by interchanging nonlocal and local operators (respectively, dense and convolutional layers) in the ROM pipeline. Finally, thanks to the use of nonlinear reduction techniques, the DL-ROM can overcome some of the difficulties implied by the Kolmogorov $n$-width. Of course, though, our approach has some limitations too. First of all, it is mesh-constrained, as it is bounded to the existence of a high-fidelity model. Secondly, it mostly relies on convolutional layers, which makes it less obvious to adapt the current implementation to non-cubic domains. Finally, the approach was originally designed for the case $p\ll\fomdim$. Even though infinite-dimensional parameters spaces can be handled as in Section \ref{sec:exp2}, better strategies may be available.}

\section{Conclusions}

We developed a novel deep learning approach for reduced order modelling of parameter dependent PDEs, here termed DL-ROM, where the solution map is approximated by a deep neural network $\Phi$. Our construction is based on the use of autoencoders, which we employ as a nonlinear alternative to other reduction techniques such as the POD. In the DL-ROM approach, we choose the latent dimension to be the smallest one granting arbitrary accuracy. The value of such dimension was investigated in detail in Section \ref{sec:nonlinear}. There, we proved some theoretical results, respectively Theorem \ref{theorem:parametric-reduction} and Theorem \ref{theorem:advediff-reduction}, that can be used as guidelines for practical applications. \review{Further insights on the potential of the DL-ROM approach were discussed in Theorem 5, Section \ref{sec:dlrom}. There, we provided explicit error estimates that were later confirmed via empirical evidence (cf. Section \ref{sec:exp1})}.

The results obtained in our experiments are promising. The DL-ROM appears to be a captivating alternative to traditional ROMs, especially in challenging situations where linear models fail. \review{Our first test case, Section \ref{sec:exp1}, shows that the method is able to handle transport-dominated problems and that it behaves well in the presence of singularities. Good results are also obtained for high-dimensional parameter spaces, Section \ref{sec:exp2}, even though it becomes harder to handle the generalization error. The latter can be either improved by increasing the number of training samples or by including physical terms in the loss function. While we wish to investigate this further in future works, we acknowledge that multiple researchers are now working on this topic, e.g. \cite{mishra, shin}.

In principle, being completely nonintrusive and data-driven, the proposed approach can be readily applied to nonlinear PDEs and more complicated systems. Also, at the cost of treating time as an additional parameter, as in \cite{fresca}, one may extend the DL-ROM approach to time dependent problems. However, some changes have to be made in order to extrapolate over time, for instance by enforcing those properties that are typical of dynamical systems (e.g. the existence of underlying semi-groups). We leave all these considerations for future works.}

We conclude with a few comments on the computational cost. While the offline stage is clearly expensive, our design choices allow for a significant reduction in the model complexity, which results in architectures that are easier to train (cf. e.g. \cite{fresca, geist}).
Nevertheless, the DL-ROM is extremely fast when used online. This makes the method suited for demanding tasks with multiple queries, as the ones typical of sensitivity analysis, uncertainty quantification and multiscale methods.

\appendix

\theoremstyle{definition}
\newtheorem{examp}[subsection]{Example}
\renewcommand{\thelemma}{A.\arabic{lemma}}
\setcounter{theorem}{1}

\final{\section{An example of slow decay in the Kolmogorov $n$-width}
\label{appendix:example}
\begin{examp}
\label{example:kolmogorov-decay}
Let $\Omega:=(-2,2)$. For any $x_{0}\in\Omega$, let $\delta_{x_{0}}$ be the Dirac delta distribution centered at $x_{0}$. Consider the 2-dimensional parameter space $\Theta:=\{\mup=(\mu_{1},\mu_{2})\in[-1,1]\times[0,1]\;|\;-1\le \mu_{1}-\mu_{2}\le \mu_{1}+\mu_{2}\le 1\}$, together with the differential problem below
$$
\begin{cases}
    -u'' = 2\delta_{\mu_{1}}-\delta_{\mu_{1}-\mu_{2}}-\delta_{\mu_{1}+\mu_{2}} & x\in\Omega
    \\u(-2) = u(2) = 0
\end{cases}
$$
For each $\mup\in\Theta$, the corresponding solution $\up$ is a piecewise linear function with support given by $[\mu_{1}-\mu_{2},\mu_{1}+\mu_{2}]$. In particular, $\up$ is a hat function with a peak of height $\mu_{2}$ at $x=\mu_{1}$. Also, by direct calculation, $$||\up||_{L^{2}(\Omega)} = \sqrt{\frac{2}{3}\mu_{2}^{3}}.$$ Let now $\solmanifold:=\{\up\}_{\mup\in\Theta}\subset V:=L^{2}(\Omega)$ and fix any positive $\rdim\in\mathbb{N}$. It is then easy to see that the functions 
$$v_{i,n}:=u_{\left(-1+\frac{i}{n}-\frac{1}{2n}, \frac{1}{2n}\right)},\quad i=1,...,2n$$
are mutually orthogonal in  $L^{2}(\Omega)$. As a consequence, the Kolmogorov $n$-width of $\solmanifold$ satisfies

\[
\begin{multlined}
d_{n}(\solmanifold) \ge d_{n}(\{v_{i,n}\}_{i=1}^{2n}) = \\\\
= d_{n}\left(\left\{||v_{i,n}||_{L^{2}(\Omega)}^{-1}v_{i,n}\right\}_{i=1}^{2n}\right)||v_{i,n}||_{L^{2}(\Omega)} = \frac{1}{\sqrt{2}}||v_{i,n}||_{L^{2}(\Omega)} = \frac{1}{2\sqrt{6}}n^{-3/2}, 
\end{multlined}
\]
\newline
where the second last equality follows by noticing that the set $\{||v_{i,n}||_{L^{2}(\Omega)}^{-1}v_{i,n}\}_{i=1}^{2n}$ is isometric to the canonical basis of $\mathbb{R}^{2n}$ (see \cite{ohlberger}).

Therefore, $d_{n}(\solmanifold)$ decays with a rate of at most $n^{-3/2}$, which is relatively slow when compared to the ideal case where the parametric map is analytic and the Kolmogorov $n$-width is known to decay exponentially, $d_{n}(\solmanifold)\sim\text{e}^{-\gamma n}$.
\end{examp}}

\section{Proof of the Claims in Section \ref{sec:experiments}}
\label{appendix:exp}
\newcommand{\xmu}{\x_{\mup}}
\noindent\textbf{Proof of Claim 1.} Let $\mup\in\Theta$. For the sake of brevity, define $\xmu:=(\mu_{6},\mu_{7})\in\Omega$. We shall recall that, by Morrey's embedding theorem \cite{evans}, we have $W^{1,4}_{0}(\Omega)\hookrightarrow\mathcal{C}^{0,1/2}(\Omega)$, the latter being the space of $1/2$-Hölder maps. As a consequence, for any $w\in W^{1,4}_{0}(\Omega)$, we have
\[
\begin{multlined}
\left|w(\xmu)-\int_{\Omega}f_{\mup}(\z)^{\epsilon}w(\z)d\z\right| = \left|\int_{\Omega}(w(\xmu)-w(\z))f_{\mup}^{\epsilon}(\z)d\z\right|\le
\\
\le \int_{\Omega}|w(\xmu)-w(\z)|\forcep^{\epsilon}(\z)d\z\le C'||w||_{W^{1,4}_{0}(\Omega)}\int_{\mathbb{R}^{2}}|\xmu-\z|^{1/2}\forcep^{\epsilon}(\z)d\z
\end{multlined}
\]
for a constant $C'>0$ independent on both $w$ and $\mup$. The change of variables $\y:=(\z-\xmu)/\epsilon$ then yields
\[
\begin{multlined}
\left|w(\xmu)-\int_{\Omega}f_{\mup}(\z)^{\epsilon}w(\z)d\z\right| \le..\le C'||w||_{W^{1,4}_{0}(\Omega)}\epsilon^{1/2}\int_{\mathbb{R}^{2}}|\y|^{1/2}G(\y)d\y
\end{multlined}
\]
where $G$ is the probability density of the standard normal distribution in $\mathbb{R}^{2}$. By passing at the supremum over $w$ with $||w||_{W^{1,4}_{0}(\Omega)}=1$ and $\mup\in\Theta$ we get
$$\sup_{\mup\in\Theta}||\forcep-\forcep^{\epsilon}||_{W^{-1,4}(\Omega)}\le C''\epsilon^{1/2}$$
for some constant $C''>0$. By classical stability estimates for elliptic PDEs, see e.g. Lemma \ref{lemma:abstract-continuity}, we then have $\sup_{\mup\in\Theta}||\up-\up^{\epsilon}||_{W^{1,4/3}(\Omega)}\le 10C''\epsilon^{1/2}$, as $\condp(\x)\ge10^{-1}$ for all $\mup\in\Theta$. Up to the embedding the solution manifold in $L^{2}(\Omega)$, the claim now follows.\qed
\newline\newline
\textbf{Proof of Claim 2}. The idea is to re-parametrize the solution manifold, as the given parametrization suffers from the lack of injectivity. In fact, both $\mu_{5}=0$ and $\mu_{5}=2\pi$ return the same advective field (and we cannot exclude one extreme, or $\Theta$ would lose its compactness).
\newcommand{\muone}{\mup'}
\newcommand{\mutwo}{\mup''}
To do so, let $S^{1}$ be the unit circle in $\mathbb{R}^{2}$. We define the hypercylinder $\Theta':=[0,1]^{4}\times S^{1}\times[0.1, 0.9]^{2}$. We will adopt a seven component notation as before, even though $\Theta'\subset\mathbb{R}^{8}$, as $\mup_{5}\in S^{1}$ is now 2-dimensional. We re-parametrize the coefficients of the PDE in terms of this new coordinates in the obvious way, especially for $\cond_{\muone}$ and $f_{\muone}$. For the advective field we let $\adve_{\muone}:=\mup_{5}$. We shall now prove that: (i) the new parameter space satisfies $\mindim(\Theta')=7$, (ii) the new parametric map $\muone\to u_{\muone}$ is continuous and (iii) injective. Claim 2 then follows by Theorem \ref{theorem:parametric-reduction}.
\newline\newline
\textit{Proof that $\mindim(\Theta')=7$}. Consider the map $\phi:\Theta'\to\mathbb{R}^{7}$ given by
$$\phi(\muone)=(\muone_{5}(1+\mu'_{1}),\; \mu'_{2},\;\mu'_{3},\;\mu'_{4},\;\mu'_{6},\;\mu'_{7}).$$
Then the image $\phi(\Theta')=\{\z\in\mathbb{R}^{2}:1\le|\z|\le2\}\times[0,1]^{3}\times[0.1,0.9]^{2}\subset\mathbb{R}^{7}$ has nonempty interior. In particular, $\mindim(\phi(\Theta'))=7$. Since $\phi$ clearly admits a continuous inverse, $\phi^{-1}: \phi(\Theta')\to\Theta'$, we conclude that $\mindim(\Theta')=7$.
\newline\newline\textit{Proof that the parametric map $\mup'\to u_{\mup'}$ is continuous}. Clearly $\cond_{\muone}$ and $\adve_{\muone}$ depend continuously on $\muone$. Using again the embedding $W^{1,4}_{0}(\Omega)\hookrightarrow\mathcal{C}^{0,1/2}(\Omega)$ as in the proof of Claim 1, it is also easy to see that the map $\muone\to f_{\muone}$ is $\Theta'\to W^{-1,4}(\Omega)$ Hölder continuous. By composition (see Lemma \ref{lemma:pde-continuity}), we then obtain the continuity of the parametric map. 
\newline\newline\textit{Proof that the parametric map $\mup'\to u_{\mup'}$ is injective}. Let $\muone,\mutwo\in\Theta'$ and assume that $u\in W^{1,4/3}(\Omega)$ is a solution for both parameters, that is $u=u_{\muone}=u_{\mutwo}$. Classical results on inner regularity of solutions to elliptic PDEs ensure that $u_{\muone}$ is locally $H^{1}$ at all points except at the location of the Dirac delta $f_{\mup'}$. The analogue holds for $u_{\mutwo}$, so clearly it must be $\mu'_{6}=\mu''_{6}$ and $\mu'_{7}=\mu''_{7}$ in order for the solutions to coincide. Next, let $w\in\mathcal{C}^{\infty}_{0}(\Omega_{0})$ and extend it to zero on $\Omega\setminus\Omega_{0}$. Using $w$ as test function for the equations of both $\muone$ and $\mutwo$ and then subtracting term by term yields
$$C\int_{\Omega_{0}}\left(\adve_{\muone}-\adve_{\mutwo}\right)\cdot\nabla u w = 0.$$
As $w$ is arbitrary, it follows that $\nabla u$ is orthogonal to $(\adve_{\muone}-\adve_{\mutwo})$ on $\Omega_{0}$. In particular, if $\adve_{\muone}\neq\adve_{\mutwo}$, then $u$ must be constant along the direction $(\adve_{\muone}-\adve_{\mutwo})$ within $\Omega_{0}$. But, because of the boundary conditions, this would make $u$ identically constant near at least one edge of $\partial\Omega$. However, this is a contradiction. In fact, $u_{|\partial\Omega}\equiv1$, thus classical maximum principles ensure that $u>1$ a.e. in $\Omega$ (see e.g. Lemma \ref{lemma:positivity}). It follows that $\adve_{\muone}=\adve_{\mutwo}$ and so $\muone_{5} = \mutwo_{5}$. We now notice that, by subtracting the equations for $\muone$ and $\mutwo$, we have
\[\begin{multlined}
\sum_{i=1}^{4}(\mu'_{i}-\mu''_{i})\int_{\Omega_{i}}\nabla u\cdot\nabla w
= 0\quad\forall w\in \mathcal{C}^{\infty}_{0}(\Omega).
\end{multlined}\]
Fix any $i\in\{1,2,3,4\}$ and let $v\in\mathcal{C}^{\infty}(\Omega_{i})$. Define $w\in\mathcal{C}^{\infty}(\Omega_{i})$ to be any of the strong solutions to the PDE $-\Delta w = v$ with homogeneous Neumann boundary condition on $\partial\Omega_{i}$. Since the subdomains are clearly separated, it is possible to extend $w$ on the whole domain $\Omega$ so that $w$ is still smooth but also vanishes on $\partial\Omega$ and on $\Omega_{j}$ for all $j\neq i$. Using such $w$ in the last identity above and integrating by parts yields
$$0 = (\mu'_{i}-\mu''_{i})\int_{\Omega_{i}}\nabla u\cdot\nabla w = (\mu'_{i}-\mu''_{i})\int_{\Omega_{i}}u(-\Delta w) = (\mu'_{i}-\mu''_{i})\int_{\Omega_{i}}uv. $$
Now assume that $\mu'_{i}\neq\mu''_{i}$. Then $\int_{\Omega_{i}}uv=0$ for all $v\in\mathcal{C}^{\infty}(\Omega_{i})\implies u_{|\Omega_{i}}\equiv0$, contradiction. Then $\mu'_{i}=\mu''_{i}$ and thus $\muone=\mutwo$, as claimed.\qed

\section{Auxiliary results on Partial Differential Equations}
\label{appendix:pdes}

\newcommand{\normv}[1]{||#1||_{V}}
\newcommand{\normw}[1]{||#1||_{W}}
\newcommand{\normdual}[1]{||#1||_{*}}
\newcommand{\bilinear}{\mathfrak{B}}
\newcommand{\bilinearc}{\mathfrak{B}_{c}}
\newcommand{\normb}[1]{|||#1|||}

\begin{lemma}
\label{lemma:abstract-continuity}
Let $(V, \normv{\cdot})$ and $(W, \normw{\cdot})$ be two Banach spaces, with $W$ reflexive. Let $(W^{*}, \normdual{\cdot})$ be the dual space of W and define $(\bilinear(V, W), \normb{\cdot})$ as the normed space of bounded bilinear forms $V\times W\to\mathbb{R}$, where 
$$|||a|||:=\sup_{\substack{\\\normv{v}=1\\\normw{w}=1}}\;|a(v,w)|.$$
Let $\bilinearc(V, W)\subset\bilinear(V,W)$ be the subset of coercive bounded bilinear forms, i.e. $a\in\mathfrak{B}(V,W)$ for which
$$\lambda(a) := \inf_{\normv{v}=1}\sup_{\normw{w}=1}\;|a(v,w)| > 0,\quad\quad\textnormal{and}\quad \inf_{\normw{w}=1}\sup_{\normv{v}=1}\;|a(v,w)|>0.$$
Then,
\begin{itemize}
    \item[i)] $\lambda$ is $\bilinear(V,W)\to\mathbb{R}$ continuous
    \newline 
    \item[ii)] For each $a\in\bilinearc(V,W)$ and $F\in W^{*}$ there exists a unique $u=u_{a,F}\in V$ such that $a(v,w)=F(w)$ for all $w\in W$. Furthermore, $u$ depends continuously on both $a$ and $F$. In particular:
    
    \begin{equation}
    \label{eq:abstract-continuity}
        \normv{u_{a,F}-u_{a',F'}} \le \frac{1}{\lambda(a)}\left(\normdual{F-F'}+\frac{1}{\lambda(a')}\normb{a-a'}\cdot\normdual{F'}\right)
    \end{equation}
    \newline
    for all $a,a'\in\bilinearc(V,W)$ and $F,F'\in W^{*}$.
    
\end{itemize}

\end{lemma}

\noindent
\proof{
i) Let $a,a'\in\bilinear(V,W)$. For every $v\in V$ and $w\in W$ with $\normv{v}=\normw{w}=1$ we have
$$a(v,w) =  a'(v,w) + (a-a')(v,w) \le |a'(v,w)| + \normb{a-a'}.$$
Since the above holds for both $w$ and $-w$, we actually have $|a(v,w)| \le |a'(v,w)| + \normb{a-a'}$, and thus $\lambda(a) \le \lambda(a') + \normb{a-a'}$. As the situation is symmetric in $a$ and $a'$, it follows that $|\lambda(a)-\lambda(a')|\le\normb{a-a'}$. In particular, $\lambda$ is Lipschitz-continuous.
\newline
\newline
ii) Given $a\in\mathfrak{B}_{c}(V,W)$ and $F\in W^{*}$, the existence and uniqueness of $u_{a,F}$ follow from a Banach space version of the Lions-Lax-Milgram theorem (see Lemma 3.1. in \cite{hoffmann}). Furthermore, one also has the stability estimate $\normv{u_{a,F}} \le (\lambda(a))^{-1}\normdual{F}$.

To get the inequality in \eqref{eq:abstract-continuity}, let $a,a'\in\bilinear(V,W)$, $F,F'\in W^{*}$ and $u:=u_{a,F}$, $u':=u_{a',F'}$. Then $a(u,w)=F(w)$ and $a'(u',w)=F'(w)$ for all $w\in V$. We subtract these two identities to get
\[
\begin{multlined}
a(u,w)-a'(u',w) = F(w)-F'(w)
\\
\implies a(u-u',w) = (F-F')(w) + (a'-a)(u',w).
\end{multlined}
\]
It follows that, for all $w\in W$, one has $a(u-u',w) \le \normdual{F-F'}\normw{w} + \normb{a'-a}\cdot\normv{u'}\normw{w}$. By linearity, using both $w$ and $-w$, we conclude that
$$|a(u-u',w)| \le \normdual{F-F'}\normw{w} + \normb{a-a'}\cdot\normv{u'}\normw{w}\quad\quad\forall w\in W.$$
In particular, passing at the supremum over $\normw{w}=1$ yields
$$\lambda(a)\normv{u-u'} \le \normdual{F-F'} + \normb{a-a'}\cdot \normv{u'}.$$
Now, we may apply the stability estimate for $\normv{u'}$ and divide by $\lambda(a)$ to get \eqref{eq:abstract-continuity}. Finally the latter, together with (i), shows that $u' \to u$ as soon as $a'\to a$ and $F'\to F$.\qed
}
\newline\newline
For the next Lemma, we consider the notation introduced in Section \ref{sec:nonlinear-param-pdes}.1.

\begin{lemma}
\label{lemma:pde-continuity}
Let $\Omega\subset\mathbb{R}^{d}$ be a bounded domain. Let $1<q<+\infty$ and define the conjugate exponent $q':=(q-1)^{-1}q$. For each $\sigma\in\Sigma(\Omega)$, $b\in B(\Omega)$, $f\in W^{-1,q'}(\Omega)$ and $g\in W^{1/q',q}(\partial\Omega)$ let $u=u_{\sigma,b,f,g}$ be the unique solution to the following variational problem
\[
\begin{multlined}
u\in W^{1,q}(\Omega):\\ u_{|\partial\Omega}=g\quad\textnormal{and}\quad\int_{\Omega}\cond\nabla u\cdot\nabla w + \int_{\Omega}\left(\adve\cdot\nabla u\right) w = \int_{\Omega}fw\quad\forall w\in W^{1,q'}_{0}(\Omega).
\end{multlined}
\]
Then, the solution map $(\cond, \adve, f, g)\to u_{\cond,\adve,f,g}$ is: (i) continuous, (ii) Lipschitz continuous on all compact subsets.
\end{lemma}
 
\noindent\proof{Before moving the actual proof, we shall recall that there exists a bounded linear operator $T: W^{1/q',q}(\partial\Omega)\to W^{1,q}(\Omega)$ for which $Tg_{|\partial\Omega} = g$, namely a right-inverse of the trace operator (see \cite{necas}). In particular, there exists a constant $\tilde{C}>0$ such that $||Tg||_{W^{1,q}(\Omega)}\le\tilde{C}||g||_{W^{1/q',q}(\partial\Omega)}$.
\newline\newline
For the sake of brevity, we let $V:= W_{0}^{1,q}(\Omega)$, $W:=W_{0}^{1,q'}$ and define $W^{*}$ as the dual space of $W_{0}^{1,q'}$ endowed with the operator norm.
As in Lemma \ref{lemma:abstract-continuity}, we also let $\bilinear(V,W)$ be the collection of all bounded bilinear maps $V\times W\to\mathbb{R}$ equipped with the corresponding operator norm. Similarly, we define $\bilinearc(V, W)$ to be the subset of coercive bounded bilinear maps. We introduce the following operators:
\begin{align*}
\mathcal{A}:\quad&L^{\infty}(\Omega,\mathbb{R}^{d\times d})\times L^{\infty}(\Omega,\mathbb{R}^{d})\;\longrightarrow \;\bilinear(V,W)\\
&(\cond,\adve)\;\longrightarrow\; a_{\cond,\adve}:=\mathcal{A}(\cond,\adve),
\end{align*}
\begin{align*}
\mathcal{F}:\quad&L^{\infty}(\Omega,\mathbb{R}^{d\times d})\times L^{\infty}(\Omega,\mathbb{R}^{d})\times W^{-1,q'}(\Omega)\times W^{1/q',q}(\partial\Omega)\;\longrightarrow \;W^{*}\\
&(\cond,\adve, f, g)\;\longrightarrow\; F_{\cond,\adve,f,g}:=\mathcal{F}(\cond,\adve,f,g),
\end{align*}
where,
$$a_{\cond,\adve}(v,w):=\int_{\Omega}\cond\nabla v\cdot\nabla w + \int_{\Omega}\left(\adve\cdot\nabla v\right) w,$$
$$
F_{\cond,\adve,f,g}(w):=\int_{\Omega}\cond\nabla Tg\cdot\nabla w + \int_{\Omega}(\adve\cdot Tg)w + \int_{\Omega}fw.
$$
We claim that:
\begin{itemize}
    \item [1)] The operator $\mathcal{A}$ is linear and continuous. Also, $a_{\cond,\adve}\in\bilinearc(V,W)$ for all choices of $\cond\in\Sigma(\Omega)$ and $\adve\in B(\Omega)$.\\
    \item [2)] The operator $\mathcal{F}$ continuous. Also, it is Lipschitz continuous when restricted to any compact subset of its domain.
\end{itemize}

\noindent We shall now prove these claims. First of all, let $C>0$ be the Poincàre constant for the domain $\Omega$ and the exponent $q'$. Then, it is straightforward to see that
\newcommand{\normq}[1]{||#1||_{W^{1,q}_{0}(\Omega)}}
\newcommand{\normqp}[1]{||#1||_{W^{1,q'}_{0}(\Omega)}}
\newcommand{\normdd}[1]{||#1||_{L^{\infty}(\Omega,\mathbb{R}^{d\times d})}}
\newcommand{\normd}[1]{||#1||_{L^{\infty}(\Omega,\mathbb{R}^{d})}}
\[
\begin{multlined}
|a_{\cond,\adve}(v,w)| \le \normdd{\cond}\normq{v}\normqp{w}\\ + C\normd{\adve}\normq{v}\normqp{w},
\end{multlined}
\]
for all $v\in V$ and $w\in W$. In particular, $\mathcal{A}$ is both linear and bounded, thus continuous. Let now $\cond\in\Sigma(\Omega)$, $\adve\in B(\Omega)$ and define $\varepsilon=\varepsilon(\cond)>0$ to be the ellipticity constant of $\cond$. We notice that if $\varphi\in\mathcal{C}^{\infty}_{0}(\Omega)$, then $\varphi$ is both an element of $V$ and $W$. Also, integrating by parts yields
\[
\begin{multlined}
a_{\cond,\adve}(\varphi,\varphi)=\int_{\Omega}\cond\nabla \varphi\cdot\nabla \varphi+\int_{\Omega}\adve\cdot\left(\varphi\nabla\varphi\right)
\ge\\\ge\varepsilon\normq{\varphi}\normqp{\varphi} + \int_{\Omega}\adve \cdot\nabla\left(\frac{1}{2}\varphi^{2}\right) =\\=\varepsilon\normq{\varphi}\normqp{\varphi} - \frac{1}{2}\int_{\Omega}\text{div}(\adve)\varphi^{2}=\\=\varepsilon\normq{\varphi}\normqp{\varphi},
\end{multlined}
\]
as $\adve$ is divergence free. It follows that for each $\varphi\in\mathcal{C}^{\infty}_{0}(\Omega)$ with $\varphi\neq0$
$$\sup_{\\\substack{\psi\in\mathcal{C}^{\infty}_{0}(\Omega)\\\normqp{\psi}=1}} |a_{\cond,\adve}(\varphi,\psi)| \ge a_{\cond,\adve}\left(\varphi,\normqp{\varphi}^{-1}\varphi\right)\ge\varepsilon\normq{\varphi}$$
and, similarly,
$$\sup_{\\\substack{\psi\in\mathcal{C}^{\infty}_{0}(\Omega)\\\normq{\psi}=1}} |a_{\cond,\adve}(\psi,\varphi)| \ge a_{\cond,\adve}\left(\normq{\varphi}^{-1}\varphi,\varphi\right)\ge\varepsilon\normqp{\varphi}.$$
Since $a_{\cond,\adve}$ is continuous and $\mathcal{C}^{\infty}_{0}(\Omega)$ is both dense in $V$ and $W$, by the above we conclude that $a_{\cond,\adve}\in\bilinearc(V,W)$. This proves claim (1).
\newline
\newline
We now move to (2). For each $\cond,\adve,f,g$ and $w\in W$ we have\newcommand{\normtg}[1]{||#1||_{W^{1,q}(\Omega)}}
\[
\begin{multlined}
|F_{\cond,\adve,f,g}(w)|\le \normdd{\cond}\normtg{Tg}\normqp{w}\\+C\normd{\adve}\normtg{Tg}\normqp{w}\\+||f||_{W^{-1,q'}(\Omega)}\normqp{w}.
\end{multlined}
\]
In particular, for all $w\in W$ with unitary norm,
\newcommand{\normboundary}[1]{||#1||_{W^{1/q',q}(\partial\Omega)}}
\[
\begin{multlined}
|F_{\cond,\adve,f,g}(w)|\le \tilde{C}\normboundary{g}\left( \normdd{\cond}+C\normd{\adve}\right)+||f||_{W^{-1,q'}(\Omega).}
\end{multlined}
\]
From here, arguing by linearity easily yields (2).
\newline\newline
Finally, for each $\cond\in\Sigma(\Omega),\adve\in B(\Omega),f\in W^{-1,q'}(\Omega),g\in W^{1/q',q}(\partial\Omega)$ let $\tilde{u}_{\cond,\adve,f,g}\in V=W^{1,q}_{0}(\Omega)$ be the unique solution to the variational problem
$$a_{\cond,\adve,f,g}(\tilde{u}, w) = F_{\cond,\adve,f,g}(w)\quad\forall w\in W.$$
At this regard, we notice that $W = W^{1,q'}_{0}(\Omega)$ is reflexive, in fact $1<q<+\infty$ implies $1<q'<+\infty$. Therefore, by Lemma \ref{lemma:abstract-continuity}, we know that $\tilde{u}_{\cond,\adve,f,g}$ exists unique and it depends continuously (by composition) on $(\cond,\adve,f,g)$. Furthermore, as clear from inequality \eqref{eq:abstract-continuity} in Lemma \ref{lemma:abstract-continuity}, the correspondence $(\cond,\adve,f,g)\to\tilde{u}_{\cond,\adve,f,g}$ is Lipschitz continuous on every compact subset of the product space $\Sigma(\Omega)\times B(\Omega)\times W^{-1,q'}(\Omega)\times W^{1/q',q}(\partial\Omega)$. This is easily deduced by the properties of $\mathcal{A}$ and $\mathcal{F}$ as well as by the fact that compactness is preserved under continuous transformations. Finally, we notice that
$$u_{\cond,\adve,f,g}=\tilde{u}_{\cond,\adve,f,g}+Tg.$$
The conclusion follows.\qed}

\begin{lemma}
\label{lemma:positivity}
Consider the context and notation in Lemma \ref{lemma:pde-continuity}. If $g\equiv c\in\mathbb{R}$ and $f > 0$ in the distributional sense, then $u>c$ a.e. in $\Omega$.
\end{lemma}

\noindent\proof{This simply derives from maximum principles. We first prove the case $c=0$. Let $\eta\in\mathcal{C}^{\infty}_{0}(\Omega)$ be such that $\eta>0$ everywhere in $\Omega$. Let $w\in H^{1}_{0}(\Omega)$ be the solution to the following adjoint variational problem:
$$\int_{\Omega}\cond^{T}\nabla w\cdot\nabla v - \int_{\Omega}(\adve\cdot\nabla w)v = \int_{\Omega}\eta v\quad\forall v\in\mathcal{C}^{\infty}_{0}(\Omega).$$
Within this regular case, the classical maximum principle states $w>\max w_{|\Omega} = 0$ in $\Omega$, see e.g. Theorem 2 in \cite{chicco}. Now we notice that $w\in  W^{1,q'}_{0}(\Omega)$, as the PDE also admits a unique solution in that space. Thus, by density, we are allowed to consider $u$ as test function for $w$ and viceversa. Doing so and subtracting the equations for $u$ and $w$ yields
$$\int_{\Omega}\eta u = \int_{\Omega}fw,$$
since $\cond\nabla u\cdot\nabla w = \cond^{T}\nabla w\cdot \nabla u$ and the advective terms cancel out using the integration by parts formula (recall that $\adve$ is divergence free while both $u$ and $w$ vanish on $\partial\Omega$). The above shows that $\int_{\Omega}\eta u > 0$, as the right hand side is positive by hypothesis. As $\eta$ was arbitrary, we conclude that $u>0$ a.e. in $\Omega$. Let now $c\neq0$. It is elementary to see that $u=c+u_{0}$, where $u_{0}$ solves the variational problem with homogenous boundary conditions. The conclusion follows.\qed}

\bibliographystyle{amsplain}

\end{document}